\documentclass{amsart}

\usepackage{amssymb}
\usepackage{graphicx}
\usepackage{amsthm}
\usepackage{parskip}
\usepackage{float}
\usepackage{tikz}
\usepackage[foot]{amsaddr}

\makeatletter
\@namedef{subjclassname@2010}{%
  \textup{2010} Mathematics Subject Classification}
\makeatother

\numberwithin{equation}{section}

\newcommand\ZZ{\mathbb{Z}}
\newcommand\ZZT{\mathbb{Z}[\sqrt{2}]}
\newcommand\ZZP{\mathbb{Z}[\sqrt{-2p}]}

\newcommand\CC{\mathbb{C}}
\newcommand\RR{\mathbb{R}}
\newcommand\QQ{\mathbb{Q}}
\newcommand\QQP{\mathbb{Q}(\sqrt{-8p})}
\newcommand\QQT{\mathbb{Q}(\sqrt{2})}
\newcommand\Sd{\sqrt{d}}
\newcommand\FF{\mathbb{F}}
\newcommand\OO{\mathcal{O}}
\newcommand\DD{\mathcal{D}}

\newcommand\RRR{\mathcal{R}}
\newcommand\TTT{\mathcal{T}}
\newcommand\BB{\mathcal{B}}
\newcommand\oa{\overline{a}}
\newcommand\ob{\overline{b}}
\newcommand\ove{\overline{\varepsilon}}
\newcommand\opi{\overline{\pi}}
\newcommand\onu{\overline{\nu}}
\newcommand\ogamma{\overline{\gamma}}
\newcommand\ow{\overline{w}}
\newcommand\oz{\overline{z}}
\newcommand\aaa{\mathfrak{a}}
\newcommand\oaaa{\overline{\mathfrak{a}}}
\newcommand\bb{\mathfrak{b}}
\newcommand\obb{\overline{\mathfrak{b}}}
\newcommand\ccc{\mathfrak{c}}
\newcommand\uu{\mathfrak{u}}
\newcommand\mm{\mathfrak{m}}
\newcommand\ff{\mathfrak{f}}
\newcommand\nn{\mathfrak{n}}
\newcommand\dd{\mathfrak{d}}
\newcommand\pp{\mathfrak{p}}
\newcommand\opp{\overline{\mathfrak{p}}}
\newcommand\PP{\mathfrak{P}}

\newcommand\rrr{\mathfrak{r}}
\newcommand\pt{\mathfrak{t}}
\newcommand\TT{\mathfrak{T}}
\newcommand\ps{\mathfrak{s}}

\newcommand\rk{\mathrm{rk}}
\newcommand\CL{\mathrm{Cl}}
\newcommand\Vol{\mathrm{Vol}}
\newcommand\SL{\mathrm{SL}_2(\mathbb{Z})}

\newcommand\Gal{\mathrm{Gal}}
\newcommand\Norm{\mathrm{Norm}}

\newcommand\mul{\mathrm{m}}
\newcommand\Disc{\mathrm{Disc}}
\newcommand\diam{\mathrm{diam}}

\newcommand\ve{\varepsilon}

\newcommand\Kt{K_{\mathfrak{t}}}
\newcommand\Gt{G_{\mathfrak{T}}}
\newcommand\Lt{L_{\mathcal{T}}}
\newcommand\ONE{\textbf{1}}

\theoremstyle{plain} \newtheorem{theorem}{Theorem}
\theoremstyle{plain} 
\theoremstyle{plain} \newtheorem{lemma}{Lemma}
\theoremstyle{plain} 
\theoremstyle{plain} \newtheorem{conjecture}{Conjecture}
\theoremstyle{plain} 
\theoremstyle{plain} 
\theoremstyle{plain}\newtheorem{prop}{Proposition}
\theoremstyle{remark} \newtheorem*{remark}{Remark}
\theoremstyle{remark} 
\theoremstyle{remark} 
\theoremstyle{remark} 

\DeclareMathOperator*{\sumsum}{\sum\sum}
\DeclareMathOperator*{\suma}{\sum{}^{\ast}}

\title{On the $16$-rank of class groups of $\QQP$ for $p\equiv -1\bmod 4$}
\date{\today}
\author{Djordjo Milovic}
\email{dmilovic@math.ias.edu}
\address{Institute for Advanced Study, Einstein Drive, Princeton, NJ 08540, USA}

\begin{document}

\begin{abstract}
We use a variant of Vinogradov's method to show that the density of the set of prime numbers $p\equiv -1\bmod~4$ for which the class group of the imaginary quadratic number field $\QQP$ has an element of order $16$ is equal to $1/16$, as predicted by the Cohen-Lenstra heuristics.
\end{abstract}

\subjclass[2010]{11N45, 11R29, 11R44, 11N36}

\maketitle

\section{Introduction and motivation}
Let $\CL(D)$ denote the (narrow) class group of the quadratic number field $\QQ(\sqrt{D})$ of discriminant $D$, and let $h(D) := \#\CL(D)$ denote its class number. Although the class group $\CL(D)$ encodes important arithmetic information about the ring of integers in $\QQ(\sqrt{D})$, very little is known about its average behavior as $D$ varies in some natural family of discriminants. The $2$-part of $\CL(D)$ is perhaps the most accessible. In \cite{Gauss}, Gauss proved that the $2$-rank (in other words, the ``width'' of the $2$-part) is given by the formula
$$
\rk_2\CL(D) := \dim_{\FF_2}(\CL(D)/2\CL(D)) = \omega(D) - 1,
$$
where $\omega(D)$ is the number of distinct prime divisors of $D$. In particular, if $\omega(D) = 1$, then the $2$-part of $\CL(D)$ is trivial, so class groups with the simplest non-trivial $2$-parts arise from discriminants that have exactly two distinct prime divisors. We focus on one family of such discriminants, namely the family $\{-8p\}_p$, where $p$ ranges over prime numbers congruent to $-1$ modulo $4$. The $2$-part of $\CL(-8p)$ is cyclic and hence completely determined by the highest power of $2$ dividing $h(-8p)$. Therefore a natural problem is to determine, for each integer $k\geq 1$, the natural density of the set of prime numbers $p\equiv -1\bmod 4$ such that $2^k$ divides $h(-8p)$. In that vein, for each integer $k\geq 1$ and real number $X>2$, we set
$$
\rho(X; 2^k) := \frac{\#\{p\leq X:\ p\equiv -1\bmod 4,\ 2^k|h(-8p)\}}{\#\{p\leq X\}},
$$
and we define $\rho(2^k) := \lim_{X\rightarrow \infty}\rho(X; 2^k)$, if the limit exists. R\'{e}dei's \cite{Redei} and Reichardt's \cite{Reichardt} work from the 1930's implies that $4$ divides $h(-8p)$ if and only if $p\equiv -1\bmod 8$ and that $8$ divides $h(-8p)$ if and only if $p\equiv -1\bmod 16$. It now follows from the \v{C}ebotarev Density Theorem that $\rho(2^k) = 2^{-k}$ for $1\leq k\leq 3$. We prove that $\rho(16) = \frac{1}{16}$. More precisely, we prove the following equidistribution result.
\begin{theorem}\label{mainCor}
For a prime number $p\equiv -1\bmod 16$, let $e_p = 1$ if $16$ divides $h(-8p)$ and let $e_p = -1$ otherwise. Then for all $X>0$, we have
$$
\sum_{\substack{p\leq X \\p\equiv -1\ (16)}}e_p\ll X^{\frac{199}{200}},
$$
where the implied constant is absolute. In particular, the natural density of the set of prime numbers $p$ such that $p\equiv -1\bmod 4$ and such that $16$ divides $h(-8p)$ is equal to $\frac{1}{16}$.
\end{theorem}
In other words, if we define the $2^k$-rank of $\CL(D)$ to be
$$\rk_{2^k}\CL(D)= \dim_{\FF_2}(2^{k-1}\CL(D)/2^k\CL(D)),$$
Theorem~\ref{mainCor} states that the natural density of the set of prime numbers $p$ such that $p\equiv -1\bmod 4$ and such that $\rk_{16}\CL(-8p) = 1$ is equal to $\frac{1}{16}$. Aside from giving a numerical density for the $16$-rank, the main novelty of Theorem~\ref{mainCor} is that the power-saving in $X$ gives strong evidence that the behavior of the $16$-rank is different from the behavior of the lower $2$-power ranks in a very essential way. In the terminology used by Serre to describe equidistribution phenomena \cite{SerreMinerva}, the $8$-rank in families of quadratic fields of the type $\{\CL(dp)\}_{p}$, where $d$ is a fixed integer and $p$ varies among primes such that $dp$ is a discriminant, is ``motivated'' -- that is, $\rk_8\CL(dp)$ is given by the trace of the Frobenius conjugacy class of $p$ in a Galois representation associated to a motive depending only on $d$; however, the $16$-rank in the family $\{\CL(-8p)\}_{p\equiv -1\bmod 4}$ appears not to be ``motivated.'' We now explain the various consequences of Theorem~\ref{mainCor} in more detail.
\subsection{Cohen-Lenstra heuristics}
Cohen and Lenstra \cite{CohenLenstra} proposed a heuristic model to predict the average behavior of class groups. They stipulate that an abelian group $G$ occurs as the class group of an imaginary quadratic field with probability proportional to the inverse of the size of the automorphism group of $G$. Hence the cyclic group of order $2^{k-1}$ should occur as the $2$-part of the class group of an imaginary quadratic number field twice as often as the cyclic group of order $2^k$. As noted above, the $2$-part of the class group $\CL(-8p)$ is cyclic, and so it is natural to make the following conjecture.
\begin{conjecture}\label{mainConj}
For all $k\geq 1$, we have $\rho(2^k) = 2^{-k}$.
\end{conjecture}
Conjecture~\ref{mainConj} for $k\leq 3$ was implicit in the work of R\'{e}dei and Reichardt mentioned above (although the case $k=3$ first appeared explicitly in literature as a result Hasse \cite{Hasse}). Moreover, Conjecture~\ref{mainConj} is supported by strong numerical evidence (for instance, the percentages of primes $p$ that are $<10^6$ such that $p\equiv -1\bmod 4$ and such that $\rk_{2^k}\CL(-8p) = 1$ for $k = 1$, $2$, $3$, $4$, $5$, and $6$ are $50.09$, $25.06$, $12.53$, $6.40$, $3.16$, and $1.62\%$, respectively). Theorem~\ref{mainCor} gives a positive answer to Conjecture~\ref{mainConj} for the case $k = 4$.
\subsection{Methods for proving density results about the $2$-part of the class group}
R\'{e}dei \cite{Redei} showed that the $4$-rank of $\CL(D)$ is essentially determined by the quadratic residue symbols $(\frac{p_1}{p_2})$ for distinct prime divisors $p_1$ and $p_2$ of $D$. Fouvry and Kl\"{u}ners \cite{FK2, FK1B, FK1} combined this characterization with analytic and combinatorial techniques to obtain many interesting results about the $4$-rank of $\CL(D)$ and the negative Pell equation $x^2-Dy^2=-1$.
\\\\
The $8$-rank is already more subtle. The main method to prove density results for the $8$-rank has been to construct certain governing fields and apply the \v{C}ebotarev Density Theorem.  More precisely, Stevenhagen \cite{Ste1} proved that if $d$ is a non-zero integer, then there exists a normal extension $M_d/\QQ$ such that the $8$-rank of $\CL(dp)$ (when $dp$ is a fundamental discriminant) is determined by the Artin conjugacy class $(p, M_d/\QQ)$ in $\Gal(M_d/\QQ)$. Knowing such a governing field $M_d$ explicitly makes it easy to study the density of primes $p$ for which $\rk_8\CL(dp) = r$ for any fixed integer~$r$. For instance, a governing field for the $8$-rank in the family $\{\CL(-8p)\}_{p\equiv -1(4)}$ is $\QQ(\zeta_{16})$, where $\zeta_{16}$ is a primitive $16$th root of unity.
\\\\
Cohn and Lagarias \cite{CohnLag} made the bold conjecture that governing fields $M_{d, 2^k}$ for the $2^k$-rank in the family $\{\CL(dp)\}_p$ as above should exist for every $2$-power $2^k$ (see also \cite{CohnLag2}). However, a governing field has \textit{not} been found for the $16$- or higher $2$-power ranks in \textit{any} family. This is the main reason that density results for the $16$-rank have been out of reach for such a long time (see \cite[p.\ 16-18]{Ste2}).
\\\\
Instead of exhibiting a governing field for the $16$-rank in the family $\{\CL(-8p)\}_p$, we introduce another method to the study of the $2$-part of class groups of quadratic number fields.

\subsection{Vinogradov's method}
In late 1940's, I.M.\@ Vinogradov \cite{Vino, Vino2} was able to prove cancellation in the sum over primes
$$
\sum_{p\leq X}\exp(2\pi i\sqrt{p})\log p
$$
by expanding it into sums of type I
$$
\sum_{n\leq X,\ n\equiv 0\bmod d} a_n
$$
and sums of type II
$$
\sumsum_{m\leq M, n\leq N}\alpha_m\beta_n a_{mn},
$$
where $a_n = \exp(2\pi i\sqrt{n})$, $d$ is any positive integer, and $\{\alpha_n\}_n$ and $\{\beta_m\}_m$ are very general sequences of complex numbers that do not grow too quickly. We prove Theorem~\ref{mainCor} by using a modern version of Vinogradov's method developed by Friedlander, Iwaniec, Mazur, and Rubin \cite[Proposition 5.2, p.722]{FIMR}.
\\\\
One important feature of Vinogradov's estimates is that the bound for the sum over primes saves a power of $X$, i.e., there exists a small real number $\delta>0$ such that 
$$
\sum_{p\leq X}a_p\log p\ll X^{1-\delta}.
$$
On the other hand, the best zero-free regions of classical $L$-functions generally give the much worse error estimates of the type $X\log(-c\sqrt{\log X})$. This suggests that the $a_p$ are not ``motivated'' -- they do not arise naturally as coefficients of a finite sum of classical $L$-functions, and in particular are not naturally related to an Artin symbol of $p$ in a fixed normal extension $M/\QQ$. In other words, the proof of Theorem~\ref{mainCor} gives strong evidence that a governing field for the $16$-rank in the family $\{\CL(-8p)\}_{p\equiv -1(4)}$ in fact does \textit{not} exist. For a more precise discussion of this phenomenon, see Section~\ref{Countingprimes}.
\subsection{A few words about the proof}
Leonard and Williams \cite{LW82} found the following criterion for the $16$-rank. Since we were unable to verify their proof of this criterion, we give another proof of a slightly more general statement in Section \ref{Governingsymbols}. A prime $p\equiv -1\bmod 16$ can be written as 
\begin{equation}\label{puv}
p = u^2-2v^2
\end{equation}
where $u$ and $v$ are integers, $u>0$, and 
\begin{equation}\label{u16}
u\equiv 1\bmod 16.
\end{equation}
Given such a representation, \cite[Theorem 3, p.205]{LW82} (or our Proposition~\ref{generalLW}) states that 
\begin{equation}\label{div16}
e_p = \left(\frac{v}{u}\right),
\end{equation}
where $e_p$ is defined as in Theorem~\ref{mainCor} and $\left(\frac{\cdot}{\cdot}\right)$ is the Jacobi symbol. The first few primes satisfying the above criterion are 127, 223, 479, 719, \ldots. Note that integers $u>0$ and $v$ satisfying \eqref{puv} and \eqref{u16} are \textit{not} unique. Nonetheless, the criterion \eqref{div16} is valid for \textit{any} choice of integers $u>0$ and $v$ satisfying \eqref{puv} and \eqref{u16}. Hence Theorem~\ref{mainCor} is a corollary of the following theorem, which we will prove using Vinogradov's method.
\begin{theorem}\label{mainThm}
For every $\epsilon>0$, there is a constant $C_{\epsilon}>0$ depending only on $\epsilon$ such that for every $X\geq 2$, we have
$$\left|\sum_{\substack{p\leq X\\ p\equiv -1\bmod 16}}\left(\frac{v}{u}\right)\right| \leq C_{\epsilon}X^{\frac{149}{150}+\epsilon},$$
where, for each prime $p$ in the sum above, $u$ and $v$ are taken to be integers satisfying \eqref{puv} and \eqref{u16}.
\end{theorem}
Theorem~\ref{mainThm} is an equidistribution result reminiscent of \cite[Theorem 2, p.\ 948]{FI1}. In~\cite{FI1}, Friedlander and Iwaniec associate a \textit{spin symbol} (i.e., a quantity taking values in $S^1\subset \CC^{\times}$) to each non-zero ideal in the Gaussian integers $\ZZ[i]$ and show, also using Vinogradov's method, that its value is equidistributed over prime ideals in $\ZZ[i]$ ordered by their norms.
\\\\
To apply Vinogradov's method to the sum $\sum_{p\leq X}e_p$, the most important task is to define a sequence $\{e_n\}_n$ in a way that one can prove good estimates for sums of type I and type II. Generalizing the proof in \cite{FI1} to our setting is made difficult by the fact that an odd ideal in the quadratic ring $\ZZT$ does not have a canonical generator -- the group of units $\ZZT^{\times}$ is infinite. We resort to averaging over four carefully chosen generators to define an analogous spin symbol. Proving that the resulting quantity is well-defined already requires significant new ideas. Proposition~\ref{generalLW2} in Section \ref{Governingsymbols} is a key result in this direction; it describes the twisting of $\left(\frac{v}{u}\right)$ by the fundamental unit $1+\sqrt{2}$. Section \ref{Governingsymbols} also contains the class field theoretic construction of the \textit{governing (spin) symbol} $e_p = \left(\frac{v}{u}\right)\chi(u)$ for the $16$-rank in the family $\{\CL(-8p)\}_{p\equiv -1(4)}$ (see Proposition~\ref{generalLW}). In Section \ref{Sumsoverprimes}, we construct spin symbols that both encode behavior of the $16$-rank in our family and are conducive to analytic techniques (see Equations \eqref{WGS} and \eqref{Wann}). We also reduce Theorem~\ref{mainThm} to a purely analytic statement (see Theorem~\ref{mainThm2}) that can be attacked by the machinery of Friedlander, Iwaniec, Mazur, and Rubin (see Proposition~\ref{sop}). The goal of Section~\ref{Fundamentaldomains} is to construct convenient fundamental domains for the multiplicative action of a fundamental unit $1+\sqrt{2}$ on $\ZZ[\sqrt{2}]$. In Section \ref{Linearsums}, we use a Polya-Vinogradov-type estimate to give bounds for sums of type I for the spin symbol. In Section \ref{Bilinearsums}, we give bounds for sums of type II of the spin symbol, thus completing the proof of Theorem~\ref{mainThm}. In the final section, we discuss the implications of the power-saving bound in Theorem~\ref{mainCor} on the existence of governing fields for the $16$-rank.
\subsection{Generalizations}
While it would be desirable to prove density results about the $16$-rank in any family of the type $\{\CL(dp)\}_p$ with $d$ fixed and $p$ varying, there are serious technical limitations on both algebraic and analytic sides of the problem. On the algebraic side, the $2$-part of $\CL(dp)$ might no longer be cyclic, and hence one would have to account for the possible interactions between the spin symbols arising from different prime divisors of $d$. On the analytic side, one would have to account for the possibility that the class group of $\QQ(\sqrt{d})$ need not be trivial.
\\\\
Perhaps an even more basic problem is that applying Vinogradov's method in this setting generally requires one to carry out analytic estimates over a number ring instead of $\ZZ$, and many such estimates require bounds on incomplete character sums that are well beyond anything currently available. For instance, similar proofs of density results about the $16$-rank in the families $\{\CL(-8p)\}_{p\equiv 1(4)}$ and $\{\CL(-4p)\}_{p}$ would require Burgess-type estimates for short character (modulo $q$) sums of length~$q^{\frac{1}{8}-\epsilon}$.
\\\\
Theorem~\ref{mainCor} thus lives on the very edge of unconditional density results about the $2$-part of class groups of quadratic number fields. So while the statement of our main theorem is not as general as one might hope for, our work nevertheless demonstrates two important ideas: that yet another classical analytic method is applicable to modern problems concerning class groups; and that the nature of the $16$-rank is of a type not seen before in the study of the $2$-part of class groups.   

\subsection*{Acknowledgments} 
I would like to give special thanks to my PhD advisors \a'{E}tienne Fouvry and Peter Stevenhagen for checking my work and helping me improve some arguments. I would also like to thank Hendrik Lenstra, Peter Sarnak, and Marco Streng for useful discussions. This research was partially supported by an ALGANT Erasmus Mundus Scholarship and by National Science Foundation agreement No. DMS-1128155.

\section{Governing symbols}\label{Governingsymbols}
The purpose of this section is to generalize \cite[Theorem 3, p.205]{LW82} and to develop a framework conducive to the analytic techniques of Friedlander, Iwaniec, Mazur, and Rubin \cite{FIMR}.
\\\\
Let $\chi$ be a character $(\ZZ/16\ZZ)^{\times}\rightarrow \CC^{\times}$ with kernel $\{\pm 1\}$. In other words, we have $\chi(\pm 1 \bmod 16) = 1$ and $\chi(\pm 7\bmod 16) = -1$. Then our generalization of \cite[Theorem 3, p.205]{LW82} is as follows: 
\begin{prop}\label{generalLW}
Let $p\equiv -1\bmod 16$ be a prime number. Let $u$ and $v$ be integers such that $p = u^2-2v^2$ and such that $u>0$ and $v\equiv 1\bmod 4$. Then
\begin{equation}
\rk_{16}\CL(-8p) = 1 \Longleftrightarrow \left(\frac{v}{u}\right)\chi(u) = 1.
\end{equation}
\end{prop}
The choice of $u$ and $v$ in the proposition above is \textit{not} unique. Let
$$\ve = 1+\sqrt{2}$$
be a \textit{fundamental unit} in $\ZZT$, so that the group of units $\ZZT^{\times}$ is generated by $\ve$ and $-1$. As the norm of $\ve$ is $-1$, the norm of $\ve^2 = 3+2\sqrt{2}$ is $1$. Let $p\equiv -1\bmod 16$ be a prime number as in Proposition~\ref{generalLW}. Given \textit{one} integer solution $(u, v) = (u_0, v_0)$ to the system
\begin{equation}\label{system}
\begin{cases}
p = u^2-2v^2 \\
u>0, v\equiv 1\bmod 4
\end{cases},
\end{equation}
then the complete set of integer solutions $(u, v)$ to the system \eqref{system} is of the form
$$u+v\sqrt{2} = \ve^{2k}(u_0+v_0\sqrt{2})$$
for some integer $k$. An interesting consequence of Proposition~\ref{generalLW} is that the quantity $\left(\frac{v}{u}\right)\chi(u)$ is \textit{independent} of the choice of $u$ and $v$ satisfying \eqref{system}. For a prime $p\equiv -1\bmod 16$, we can thus define the \textit{governing symbol} for the $16$-rank to be 
\begin{equation}\label{govp}
\left\langle p\right\rangle:= \left(\frac{v}{u}\right)\chi(u),
\end{equation}
where $u$ and $v$ are integers satisfying \eqref{system}. The quantity $\left\langle p \right\rangle$ determines the $16$-rank of the class group $\CL(-8p)$. It is interesting to note that the $16$-rank of $\CL(-8p)$ depends on a ``quantitative'' aspect of the splitting behavior of $p$ in $\ZZT$ that appears to allow no description purely in terms of the ``qualitative'' splitting behavior of $p$ in some normal extension of $\QQ$.
\\\\
Leonard and Williams claim that \cite[Theorem 3, p.205]{LW82} can be proved by numerous manipulations of Jacobi symbols and applications of quadratic reciprocity. We instead prove Proposition~\ref{generalLW} by interpreting the Jacobi symbol $\left(\frac{v}{u}\right)$ as an Artin symbol of an ideal $\uu$ defined via the decomposition $p = u^2-2v^2$ in an extension of $\QQP$ defined via the \textit{same} decomposition $p = u^2-2v^2$. Moreover, a by-product of our proof is the following proposition, which turns out to be essential for a successful application of the analytic tools we wish to use. 
\begin{prop}\label{generalLW2}
Let $u_1$ and $v_1$ be integers such that $u_1$ is odd and positive and such that $u_1^2-2v_1^2>0$. Define integers $u_2$ and $v_2$ by the equality
$$u_2+v_2\sqrt{2} = \ve^8(u_1+v_1\sqrt{2}).$$
Then
$$
\left(\frac{v_1}{u_1}\right) = \left(\frac{v_2}{u_2}\right).
$$
In other words, we have the equality of Jacobi symbols
$$
\left(\frac{v_1}{u_1}\right) = \left(\frac{408u_1+577v_1}{577u_1+816v_1}\right). 
$$
\end{prop}
The rest of this section is devoted to proving Proposition~\ref{generalLW} and Proposition~\ref{generalLW2}.

\subsection{Preliminaries}
\subsubsection{Galois theory}
We will make extensive use of the following lemma from Galois theory (see \cite[Chapter VI, Exercise 4, p.321]{Lang}). 
\begin{lemma}\label{lemNor}
Let $F$ be a field of characteristic different from $2$, let $E = F(\sqrt{d})$, where $d\in F^{\times}\setminus (F^{\times})^2$, and let $L = E(\sqrt{x})$, where $x\in E^{\times}\setminus (E^{\times})^2$. Let $N = \Norm_{E/F}(x)$. Then we have three cases:
\begin{enumerate}
	\item If $N\notin (E^{\times})^2\cap F^{\times} = (F^{\times})^2\cup d\cdot (F^{\times})^2$, then $L/F$ has normal closure $L(\sqrt{N})$ and $\Gal(L(\sqrt{N})/F)$ is a dihedral group of order $8$.
	\item If $N\in (F^{\times})^2$, then $L/F$ is normal and $\Gal(L/F)$ is a Klein four-group.
	\item If $N\in d\cdot (F^{\times})^2$, then $L/F$ is normal and $\Gal(L/F)$ is a cyclic group of order $4$. 
\end{enumerate}
\end{lemma}

\subsubsection{The Artin map and Artin symbols}
Let $E/F$ be a finite abelian extension of number fields. Let $I_F$ denote the free abelian group generated by prime ideals of $F$ that are unramified in $E$. The \textit{Artin map} is the group homomorphism
$$\left(\frac{\cdot}{E/F}\right): I_F\rightarrow \Gal(E/F)$$
defined as follows. Let $\pp$ be a prime ideal of $F$ which is unramified in $E$ and let $\PP$ be any prime ideal of $E$ lying above $\pp$. Let $\Norm(\pp)$ be the cardinality of the residue field at $\pp$. Then the \textit{Artin symbol} 
$$\left(\frac{\pp}{E/F}\right)$$
is the unique element of $\Gal(E/F)$ such that
$$\left(\frac{\pp}{E/F}\right)(\alpha) \equiv \alpha^{\Norm(\pp)}\bmod \PP$$
for all $\alpha$ in $E$. We then extend $\left(\frac{\cdot}{E/F}\right)$ multiplicatively to $I_F$.
\\\\
We will use the following lemma several times.
\begin{lemma}\label{lemRed}
Let $E/F$ be an abelian extension of number fields, let $L/F$ be a finite extension, and let 
$$\iota:\Gal(EL/L)\hookrightarrow\Gal(E/F)$$
be the restriction-to-$E$ map. Then for every prime ideal $\pp$ of $L$ that is coprime to $\Disc(E/F)$, we have
$$
\iota\left(\frac{\pp}{EL/L}\right) = \left(\frac{\Norm_{L/F}(\pp)}{E/F}\right).
$$
\end{lemma}
\begin{proof}
See \cite[Proposition 3.1, p. 103]{Janusz}.
\end{proof}

\subsubsection{$2^n$-Hilbert class fields}
Let $K = \QQP$ and $\CL = \CL(-8p)$. Recall that the \textit{Hilbert class field} $H$ of $K$ is the maximal unramified abelian extension of $K$. The Artin symbol induces a canonical isomorphism of groups
\begin{equation}\label{ArtSym}
\left(\frac{\cdot}{H/K}\right):\CL\longrightarrow \Gal(H/K). 
\end{equation}
Suppose for the moment that $\rk_{2^n}\CL(-8p) = 1$. Then $2^n\CL$ is a subgroup of $\CL$ of index $2^n$. We define the $2^n$-\textit{Hilbert class field} $H_{2^n}$ to be the subfield of $H$ fixed by the the image of $2^n\CL$ under the isomorphism \eqref{ArtSym}. Since the $2$-primary part of $\CL$ is cyclic, it follows immediately that $H_{2^n}$ is the unique unramified, cyclic, degree-$2^n$ extension of $K$. Moreover, \eqref{ArtSym} induces a canonical isomorphism of cyclic groups of order $2^n$
\begin{equation}\label{ArtSym2n}
\left(\frac{\cdot}{H_{2^n}/K}\right):\CL/2^n\CL\longrightarrow \Gal(H_{2^n}/K). 
\end{equation}
The main idea of the proof of Proposition~\ref{generalLW} is to write down explicitly, for $p\equiv -1\bmod 8$, \textit{both}
\begin{itemize}
	\item the $4$-Hilbert class field $H_4$ of $K$, \textit{and} 
	\item an ideal $\uu$ generating a class of order $4$ in $\CL(-8p)$
\end{itemize}
\textit{in terms of} integers $u$ and $v$ satisfying $p = u^2-2v^2$, and then to characterize those $p$ such that 
\begin{equation}\label{atriv}
\left(\frac{\uu}{H_{4}/K}\right) = 1.
\end{equation}
The isomorphism \eqref{ArtSym2n} for $n = 2$ and the equality \eqref{atriv} then imply that the class of order $4$ in $\CL$ in fact belongs to $4\CL$, which proves that $\CL$ has an element of order $16$. 

\subsubsection{Ring class fields}
To prove Proposition~\ref{generalLW2}, we will have to work with a generalization of the Hilbert class field. Let $D<0$ be any integer $\equiv 0, 1\bmod 4$ that is not a square, and let $\OO_D$ be the quadratic order of discriminant $D$, i.e., 
$$
\OO_D = \ZZ[(D+\sqrt{D})/2].
$$
Let $K = \QQ(\sqrt{D})$ be the field of fractions of $\OO_D$. Then $K$ is an imaginary quadratic number field of discriminant $\Disc(K)$ satisfying the equality
$$
D = f^2\Disc(K)
$$
for some positive integer $f$, called the \textit{conductor} of $\OO_D$. Let $\CL(D)$ denote the class group of $\OO_D$. Then there is a unique abelian extension $R_D/K$ called the \textit{ring class field} of $\OO_D$ such that the Artin map induces a canonical isomorphism of groups
\begin{equation}\label{RCF1}
\left(\frac{\cdot}{R_D/K}\right):\CL(D)\longrightarrow \Gal(R_D/K). 
\end{equation}
In the case $f = 1$, so that $D = \Disc(K)$, the ring class field $R_D$ coincides with the Hilbert class field of $K$.
\\\\
The main property of ring class fields of imaginary quadratic orders that we will use is stated in the following lemma.
\begin{lemma}\label{RCF2}
Let $K$ be an imaginary quadratic number field of even discriminant, and let $L/K$ be a cyclic extension such that:
\begin{itemize}
	\item $L/\QQ$ is a dihedral extension, and
	\item the conductor of $L/K$ divides $(4)$.
\end{itemize}
Then $L$ is contained in the ring class field $R_{D}$ of the imaginary quadratic order $\OO_D$ of discriminant $D = 16\cdot\Disc(K)$.
\end{lemma}
\begin{proof}
See \cite[Theorem 9.18, p.\ 191]{Cox} and \cite[Exercise 9.20, p.\ 195-196]{Cox}.
\end{proof}

\subsection{A special family of quadratic fields}\label{SpecialFields}
Let $u$ and $v$ be coprime integers such that $u$ is odd and positive and such that 
\begin{equation}\label{nuv}
n = u^2-2v^2
\end{equation}
is positive as well. Let $K$ be the imaginary quadratic number field defined by 
$$
K = \QQ(\sqrt{-2n}).
$$
Note that $n \equiv \pm 1 \bmod 8$, and moreover $n\equiv 1\bmod 8$ if and only if $v$ is even. Let $m$ and $d$ be the unique positive integers such that $m$ is squarefree and such that $n = d^2m$. Then $K = \QQ(\sqrt{-2m})$ and the discriminant of $K/\QQ$ is equal to $-8m$. We emphasize that both $m$ and $d$ are odd. As $\gcd(u, v) = 1$, every prime dividing $n$ splits in $\QQT$. Hence there exist $\delta$ and $\mu$ in $\QQ(\sqrt{2})$ of norm $d$ and $m$, respectively, such that $u+v\sqrt{2} = \delta^2\mu$. 
\\\\
Let $G = K(\sqrt{2})$. Note that $G$ coincides with the \textit{genus field} of $K$ in the case that $n$ is a prime number congruent to $-1$ modulo $4$. Finally, we define a quadratic extension of $G$ as follows. Define $\nu\in\ZZT\subset G$ by setting
\begin{equation}\label{nu}
\nu = u+v\sqrt{2}.
\end{equation}
Then let $L = L_{u, v} = G(\sqrt{\ve\nu})$, where $\ve = 1+\sqrt{2}$ as before. If $n$ is a prime number congruent to $-1$ modulo $8$ and $u$ and $v$ are chosen as in the statement of Proposition~\ref{generalLW}, we will see that $L$ coincides with the $4$-Hilbert class field $H_4$ of $K$.
\begin{remark}
The fields $K$ and $G$ are determined simply by $n$. In other words, had we started with another choice of integers $u$ and $v$ giving rise to the same $n$, the definitions of $K$ and $G$ would not change. However, the field $L$ may depend on the specific choice of $u$ and $v$. Since we fixed $u$ and $v$ in the beginning of the section, this should not cause any confusion.
\end{remark}
We now introduce some notation and prove some properties of the extensions $K\subset G\subset L$. Let $\onu = u-v\sqrt{2}$ be the conjugate of $\nu$ in $\QQT$. We now state a few consequences of the assumption that $\gcd(u, v) = 1$. It will be useful to consider the following field diagram.
\begin{center} 
\begin{tikzpicture}
  \draw (0, 4.5) node[]{$L = G(\sqrt{\ve\nu})$};
  \draw (0, 3) node[]{$G = K(\sqrt{2})$};
	\draw (-3, 3) node[]{$A = \QQ(\sqrt{2}, \sqrt{\ve\nu})$};
  \draw (0, 1.5) node[]{$K = \QQ(\sqrt{-2m})$};
	\draw (-3, 1.5) node[]{$\QQ(\sqrt{2})$};
	\draw (0, 0) node[]{$\QQ$};
  \draw (0, 0.3) -- (0, 1.2);
	\draw (-0.3, 0.3) -- (-2.7, 1.2);
	\draw (0, 1.8) -- (0, 2.7);
  \draw (0, 3.3) -- (0, 4.2);
	\draw (-3, 1.8) -- (-3, 2.7);
	\draw (-2.7, 1.8) -- (-0.3, 2.7);
	\draw (-2.7, 3.3) -- (-0.3, 4.2);
\end{tikzpicture}
\end{center}

\begin{lemma}\label{P1}
The extension $L/K$ is cyclic of degree $4$, and the extension $L/\QQ$ is dihedral of order $8$. 
\end{lemma}
\begin{proof}
We have
$$
\Norm_{G/K}(\ve\nu) = \Norm_{\QQ(\sqrt{2})/\QQ}(\ve\nu) = -\nu\onu = -n.
$$
As
$$
-n = 2\cdot \left(\frac{1}{2}\sqrt{-2n}\right)^2\in 2\cdot (K^{\times})^2,
$$
the first claim follows from Lemma~\ref{lemNor}, part (3). Now let $A = \QQ(\sqrt{2}, \sqrt{\ve\nu})$. As
$$
-n \notin (\QQ^{\times})^2\cup 2\cdot (\QQ^{\times})^2,
$$
part (1) of Lemma~\ref{lemNor} implies that $L = A(\sqrt{-n})$ is the normal closure of $A/\QQ$ and $\Gal(L/\QQ)\cong D_8$.
\end{proof}
Let $\pt$ denote the prime of $K$ lying above $2$.
\begin{lemma}\label{P2}
$L/K$ is unramified at every prime other than possibly at $\pt$.
\end{lemma}
\begin{proof}
Recall that $\nu = \delta^2\mu$, so $L = \QQ(\sqrt{-2m}, \sqrt{2}, \sqrt{\ve\mu})$. As the norm of $\mu$ is $m$, every prime that ramifies in $L/\QQ$ must divide $2m$. Let $p$ be a rational prime dividing $m$. Suppose $p$ factors as $\pi\opi$ in $\ZZT$, and, without loss of generality, suppose $\pi$ divides $\onu$. As $u$ and $v$ are coprime, $\nu$ and $\onu$ are coprime in $\ZZT$ and hence $\pi$ does not ramify in $A = \QQ(\sqrt{2}, \sqrt{\ve\nu})$. Thus, as $p$ splits in $\QQT$, its ramification index in $L/\QQ$ is at most $2$. But $p$ already ramifies in $K/\QQ$, and hence every prime $\pp$ of $K$ lying above $p$ must be unramified in $L/K$.
\end{proof}
By Lemma~\ref{P2}, the only prime that can divide the conductor $\ff$ of $L/K$ is the prime~$\pt$. The following lemma gives the precise power of $\pt$ dividing $\ff$. 
\begin{lemma}\label{P3}
Let $\ff$ denote the conductor of $L/K$. Then:
\begin{enumerate}
	\item If $v\equiv 1\bmod 4$, then $L/K$ is unramified and $\ff = 1$.
	\item If $v\equiv -1\bmod 4$, then $\ff = \pt^2 = (2)$.
	\item If $v\equiv 0 \bmod 2$, then $\ff = \pt^4 = (4)$.
\end{enumerate}
\end{lemma}
\begin{proof}
Since $\pt$ is the only prime that can divide $\ff$, we only need to study the extensions locally at the primes above $2$. Let $\TT$ be a prime of $G$ lying above $\pt$ and $\TTT$ a prime of $L$ lying above $\TT$. Let $\Kt$, $\Gt$, and $\Lt$ denote the completions of $K$, $G$, and $L$ with respect to the primes $\pt$, $\TT$, and $\TTT$, respectively.
\\\\
If $v$ is odd, then $n\equiv -1\bmod 8$, and so $\Kt = \QQ_2(\sqrt{-2n}) = \QQ_2(\sqrt{2})$ and $\Gt = \Kt(\sqrt{2}) = \Kt$. Thus the extension $\Gt/\Kt$ is trivial and $\Lt = \QQ_2(\sqrt{2}, \sqrt{\ve\nu})$. The extension $\QQ_2(\sqrt{2}, \sqrt{\ve\nu})/\QQ_2(\sqrt{2})$ is unramified if and only if $\ve\nu$ is a square modulo $\pt^4$; here $\pt = (\sqrt{2})$ is the maximal ideal in $\ZZ_2[\sqrt{2}]$. If $v\equiv 1\bmod 4$, then
$$
\ve\nu = (u+2v) + (u+v)\sqrt{2} \equiv 
\begin{cases} 
      1 \bmod \pt^4& \text{ if }u\equiv -1\bmod 4, \\
      \ve^2 \bmod \pt^4& \text{ if }u\equiv 1\bmod 4,
\end{cases}
$$
and hence $\Lt/\Kt$ is unramified. This proves part (1) of the lemma. Similarly, if $v\equiv 1\bmod 4$, then 
$$
\ve\nu\equiv 3\text{ or }1+2\sqrt{2}\bmod \pt^4.
$$
In this case $\ve\nu$ is not a square modulo $\pt^4$, and so $\Lt/\Kt$ is ramified. The ramification is wild, and thus $\ff$ must be divisible by $\pt^2$. As $\ve\nu\equiv 1\bmod \pt^2$, the extension $\Lt/\Kt$ can be generated by a root of the polynomial
$$
X^2+\sqrt{2}X+\frac{1-\ve\nu}{2} = \frac{1}{2}\left(\left(\sqrt{2}X+1\right)^2-\ve\nu\right),
$$
whose discriminant is $2 \bmod \pt^4$. Hence $\ff = \pt^2$ and part (2) of the lemma is proved.
\\\\
Finally, suppose $v\equiv 0\bmod 2$, so that $n\equiv 1\bmod 8$. Then $\Kt = \QQ_2(\sqrt{-2n}) = \QQ_2(\sqrt{-2})$ and $\Gt = \Kt(\sqrt{2}) = \QQ_2(\zeta_8)$. The quadratic extension $\Gt/\Kt$ is ramified of conductor $\pt^2$, where $\pt = (\sqrt{-2})$ is the maximal ideal in $\ZZ_2[\sqrt{-2}]$. Let $s = 1+\zeta_8$ be a generator of the maximal ideal $\ps$ in $\ZZ_2[\zeta_8]$. Note that $s^2 = \sqrt{2}\cdot \zeta_8\ve$, so $\ve\nu\equiv 1\bmod \ps^2$. Hence the extension $\Lt/\Kt$ can be generated by a root of the polynomial 
$$
X^2+s^3\zeta_8^6\ve^{-2}X +\frac{1-\ve\nu}{s^2} = \frac{1}{s^2}\left(\left(s X+1\right)^2-\ve\nu\right),
$$
whose discriminant is $s^6 \bmod \ps^7$. Hence the discriminant of $\Lt/\Gt$ is $\ps^6$.
\\\\
To finish, we use the conductor-discriminant formula, i.e.,
$$
\Disc(\Lt/\Kt) = \Disc(\Gt/\Kt)\ff(\Lt/\Kt)^2.
$$
The discriminant formula for the tower of fields $\Kt\subset \Gt\subset \Lt$ gives
$$
\Disc(\Lt/\Kt) = \Disc(\Gt/\Kt)^2\Norm_{\Gt/\Kt}(\Disc(\Lt/\Gt)),
$$
so that
$$
\ff(\Lt/\Kt)^2 = \Disc(\Gt/\Kt)\Norm_{\Gt/\Kt}(\Disc(\Lt/\Gt)).
$$
Substituting $\Disc(\Gt/\Kt) = \pt^2$ and $\Disc(\Lt/\Gt) = \ps^6$ into the formula above implies that $\ff(\Lt/\Kt) = \pt^4$, which completes the proof of part (3) of the lemma.
\end{proof}

\begin{lemma}\label{P4}
$L$ is contained in the ring class field $R_{D}$ of the imaginary quadratic order $\OO_D$ of discriminant $D = 16\cdot -8m$.
\end{lemma}
\begin{proof}
Combine Lemmas \ref{RCF2}, \ref{P1}, and \ref{P3}. 
\end{proof}

\subsection{A computation of Artin symbols}\label{ArtinComp}
This section contains the heart of the proof of both Proposition~\ref{generalLW} and Proposition~\ref{generalLW2}.
\\\\
The integers $u$ and $v$ appearing in \eqref{nuv} are not unique. Given a representation $n = u^2-2v^2$, another representation can be obtained by multiplying $u+v\sqrt{2}$ by $3+2\sqrt{2}$. This transforms $(u, v)$ into $(3u +4v, 2u + 3v)$.
\\\\
We will show how the quantity $\left(\frac{v}{u}\right)\chi(u)$, where $\chi$ is a Dirichlet character from Proposition~\ref{generalLW2}, naturally arises in the computation of a certain Artin symbol. This computation is somewhat delicate because the Artin symbol will take a value in a cyclic group of order $4$, and such a group has a non-trivial automorphism.
\begin{remark}
In \cite{HKW}, Halter-Koch, Kaplan, and Williams compute Artin symbols in similar cyclic field extensions $L/K$ of degree $4$. Their results, however, involve computations of Artin symbols of ideals of $K$ of order $2$ in the class group of $K$, and hence only give information about the $8$-rank in certain quadratic fields.
\end{remark}
Let $f \in\{1, 4\}$. The case $f=1$ will be used to prove Proposition~\ref{generalLW}, while the case $f=4$ will be used to prove Proposition~\ref{generalLW2}. Let $\tau = f\sqrt{-2n}$, so that $\ZZ[\tau]$ is the order of $K$ of discriminant $-8nf^2$. We define a homomorphism 
$$
\psi_{u, v}: \ZZ[\tau] \rightarrow \ZZ/u\ZZ
$$
by sending $\tau\mapsto 2vf \bmod u$. This homomorphism is well-defined since
$$
\tau^2 = -2nf^2 = -2(u^2-2v^2)f^2 \equiv (2vf)^2\bmod u.
$$
Let
\begin{equation}\label{defu}
\uu = \ker \psi_{u, v}.
\end{equation}
It is the ideal of $\ZZ[\tau]$ generated by $u$ and $2vf-\tau$, i.e., $\uu = (u, 2vf-\tau)$. In case $n = p \equiv -1\bmod 8$ and $f=1$, the ideal class of $\uu$ turns out to have order $4$, as we will see later. We remark that
\begin{equation}\label{congu}
2vf\equiv \tau\bmod \uu
\end{equation}
and that
\begin{equation}\label{normu}
\Norm(\uu) = u.
\end{equation}
Let $\sqrt{\ve\nu}$ be a square root of $\ve\nu$. Then, by Lemma~\ref{lemNor}, the extension $G(\sqrt{\ve\nu})/K$ is cyclic of degree $4$. We are interested in computing the Artin symbol
$$\left(\frac{\uu}{G(\sqrt{\ve\nu})/K}\right).$$
The key idea is to relate this Artin symbol to the Artin symbol associated to a different but related cyclic degree-$4$ extension of $K$. Let
\begin{equation}
\gamma=(2+\sqrt{2})v \in\ZZT.
\end{equation}
Then again by Lemma~\ref{lemNor}, the extension $G(\sqrt{\gamma})/K$ is cyclic of degree $4$. The element $\gamma$ was chosen so that
\begin{equation}\label{pgcong}
\ve\nu\equiv \gamma\bmod \uu,
\end{equation}
and at the same time so that the extension $\QQ(\sqrt{\gamma})/\QQ$ mimics the cyclic degree-$4$ subextension of the cyclotomic extension $\QQ(\zeta_{16})/\QQ$. Finally, let $F$ be the compositum of $G(\sqrt{\ve\nu})$ and $G(\sqrt{\gamma})$. We have the following field diagram.
\begin{center} 
\begin{tikzpicture}
  \draw (0, 0) node[]{$K = \QQ(\sqrt{-2n})$};
  \draw (2, 2) node[]{$G = K(\sqrt{2})$};
	\draw (0, 2) node[]{$K(\sqrt{\beta'})$};
  \draw (-2, 2) node[]{$K(\sqrt{\beta})$};
	\draw (0, 4) node[]{$G(\sqrt{\ve\nu\gamma})$};
	\draw (4, 4) node[]{$G(\sqrt{\gamma})$};
	\draw (2, 6) node[]{$F = G(\sqrt{\ve\nu}, \sqrt{\gamma})$};
	\draw (2, 4) node[]{$G(\sqrt{\ve\nu})$};
  \draw (0, 0.3) -- (0, 1.7);
	\draw (0.3, 0.3) -- (1.7, 1.7);
	\draw (-0.3, 0.3) -- (-1.7, 1.7);
	\draw (0, 2.3) -- (0, 3.7);
	\draw (1.7, 2.3) -- (0.3, 3.7);
	\draw (-1.7, 2.3) -- (-0.3, 3.7);
	\draw (2, 2.3) -- (2, 3.7);
	\draw (2.3, 2.3) -- (3.7, 3.7);
	\draw (2, 4.3) -- (2, 5.7);
	\draw (3.7, 4.3) -- (2.3, 5.7);
	\draw (0.3, 4.3) -- (1.7, 5.7);
\end{tikzpicture}
\end{center}
Here $\beta$ and $\beta'$ are elements of $K$ that are conjugate over $\QQ$. Let $\overline{\ve\nu\gamma} \in\QQT$ be the conjugate of $\ve\nu\gamma$ over $\QQ$. Since 
$$
\left(\sqrt{2\ve\nu\gamma}\pm \sqrt{2\ove\onu\ogamma}\right)^2 = 4v((4u+6v)\pm\sqrt{-2n}) = \frac{4v}{f}\left((4u+6v)f\pm\tau \right),
$$
we can take
$$
\beta = v((4u+6v)f-\tau)\ \ \ \ \ \ \ \ \text{and}\ \ \ \ \ \ \ \ \beta' = v((4u+6v)f+\tau).
$$
The inclusion $\Gal(F/K(\sqrt{\beta}))\subset \Gal(F/K)$ and projections $\Gal(F/K)\twoheadrightarrow \Gal(G(\sqrt{\ve\nu})/K)$ and $\Gal(F/K)\twoheadrightarrow \Gal(G(\sqrt{\gamma})/K)$ induce canonical isomorphisms 
$$
\psi_1: \Gal(F/K(\sqrt{\beta})) \stackrel{\sim}{\longrightarrow} \Gal(G(\sqrt{\ve\nu})/K)
$$
and
$$
\psi_2: \Gal(F/K(\sqrt{\beta})) \stackrel{\sim}{\longrightarrow} \Gal(G(\sqrt{\gamma})/K).
$$
Using \eqref{congu}, we find that if $\pp$ is a prime ideal dividing $\uu$, then
$$
\left(\frac{\beta}{\pp}\right) = \left(\frac{v((4u+6v)f-\tau)}{\pp}\right) = \left(\frac{4v^2f}{\pp}\right) = 1, 
$$
and so $\pp$ splits in $K(\sqrt{\beta})$. By Lemma~\ref{lemRed}, for any prime $\PP$ of $K(\sqrt{\beta})$ lying above a prime ideal $\pp$ dividing $\uu$, we have
$$
\psi_1\left(\left(\frac{\PP}{F/K(\beta)}\right)\right) = \left(\frac{\pp}{G(\sqrt{\ve\nu})/K}\right)
$$
and
$$
\psi_2\left(\left(\frac{\PP}{F/K(\beta)}\right)\right) = \left(\frac{\pp}{G(\sqrt{\gamma})/K}\right).
$$
Multiplying over all prime ideals $\pp$ dividing $\uu$, we have proved the following key lemma.
\begin{lemma}\label{keylemma1}
Let $\uu$ be defined as in \eqref{defu}. Then
$$\psi_2\circ \psi_1^{-1}\left(\left(\frac{\uu}{G(\sqrt{\ve\nu})/K}\right)\right) = \left(\frac{\uu}{G(\sqrt{\gamma})/K}\right).$$
\end{lemma}
Now we apply Lemma~\ref{lemRed} with $E = \QQ(\sqrt{-2n})$, $F = \QQ$, and $L = \QQ(\sqrt{\gamma})$. We have
$$
\iota\left(\left(\frac{\uu}{G(\sqrt{\gamma})/K}\right)\right) = \left(\frac{u}{\QQ(\sqrt{\gamma})/\QQ}\right),
$$
so that, by Lemma~\ref{keylemma1}, we have
$$\iota\circ \psi_2\circ\psi_1^{-1}\left(\left(\frac{\uu}{G(\sqrt{\gamma})/K}\right)\right) = \left(\frac{u}{\QQ(\sqrt{\gamma})/\QQ}\right).$$
Now observe that $\QQ(\sqrt{\gamma})$ is a subfield of $\QQ(\zeta_{16}\sqrt{v})$. Indeed, $\zeta_{16}\sqrt{v} + \zeta_{16}^{-1}\sqrt{v} = \gamma$. There is a canonical isomorphism
$$
\Gal(\QQ(\zeta_{16}\sqrt{v})/\QQ)\cong (\ZZ/16\ZZ)^{\times}\cong \left\langle -1\bmod 16\right\rangle\times\left\langle 3\bmod 16\right\rangle
$$
given by sending
$$
\left(\zeta_{16}\sqrt{v} \mapsto \zeta_{16}^k\sqrt{v}\right) \mapsto (k\bmod 16).
$$
Then $\QQ(\sqrt{\gamma})$ is the subfield of $\QQ(\zeta_{16}\sqrt{v})$ fixed by $-1$. For each prime $q$ coprime to $2v$, we have
$$
\left(\frac{q}{\QQ(\zeta_{16}\sqrt{v})/\QQ}\right) = q\left(\frac{v}{q}\right) \bmod 16,
$$
so that if we identify $\psi_3: \left\langle 3\bmod 16\right\rangle\cong \mu_4 = \left\langle i\right\rangle\subset \CC^{\times}$ by sending $3\mapsto i = \sqrt{-1}$, we get
$$
\psi_3\left(\left(\frac{q}{\QQ(\sqrt{\gamma})/\QQ}\right)\right) = \left(\frac{v}{q}\right)\chi(q).
$$
Multiplying over all primes $q$ dividing $u$ and using Lemma~\ref{keylemma1}, we finally obtain the following result.
\begin{lemma}\label{keylemma2}
Let $\psi: \Gal(G(\sqrt{\ve\nu})/K)\stackrel{\sim}{\longrightarrow} \mu_4$ be the isomorphism of cyclic groups of order $4$ defined by $\psi = \psi_3\circ\iota\circ \psi_2\circ \psi_1^{-1}$. Then 
$$
\psi\left(\left(\frac{\uu}{G(\sqrt{\ve\nu})/K}\right)\right) = \left(\frac{v}{u}\right)\chi(u).
$$
\end{lemma}

\subsection{An ideal identity}\label{IdealIdentity}
We keep the same notation as in Sections \ref{SpecialFields} and \ref{ArtinComp}. Recall that $\tau = f\sqrt{-2n}$, where $f\in \{1, 4\}$. Let $\pt_f$ be the ideal of $\ZZ[\tau]$ defined as the kernel of the homomorphism
$$
\tau_f: \ZZ[\tau] \rightarrow \ZZ/2f^2\ZZ
$$
given by sending $\tau\mapsto 2vf$. The homomorphism $\tau_f$ is well-defined because
$$
\tau^2 = -2nf^2 = 4v^2f^2 - 2u^2f^2 \equiv (2vf)^2 \bmod 2f^2.
$$
Then $\pt_f = (2vf-\tau, 2f^2)$. The following identity of between ideals in $\ZZ[\tau]$ will be useful in proofs of both Proposition~\ref{generalLW} and Proposition~\ref{generalLW2}.
\begin{lemma}\label{LemmaIdealIdentity}
Let $\uu$ be defined as in \eqref{defu}. Then
$$
(2vf-\tau) = \pt_f\uu^2.
$$
\end{lemma}
\begin{proof}
The principal ideal $2vf-\tau$ is invertible of norm $2u^2f^2$. Since $u$ is odd and $\gcd(u, v) = 1$, we deduce that $\uu$ is coprime to the discriminant $-8nf^2$ of $\ZZ[\tau]$ and is thus invertible. No rational primes can divide $2vf-\tau$ and $\uu$ divides $(2vf-\tau)$ by definition, so it must be that $\uu^2$ divides $(2vf-\tau)$.
\\\\
The ideal $\pt_f$ of norm $2f^2$ contains $(2vf-\tau)$ and has the same norm as the invertible ideal $(2vf-\tau)\uu^{-2}$. Hence we must have $(2vf-\tau)\uu^{-2} = \pt_f$.
\end{proof}

\subsection{Proof of Proposition~\ref{generalLW}}
We apply the results of Sections \ref{ArtinComp} and \ref{IdealIdentity} in the case $n = p\equiv -1\bmod 8$ is a prime number and $f = 1$. In this case there exist integers $u$ and $v$ such that $p = u^2-2v^2$, and the congruence $p\equiv -1\bmod 8$ immediately implies that both $u$ and $v$ are odd. Without loss of generality, we may assume that $u$ is positive and 
\begin{equation}\label{vcong}
v\equiv 1\bmod 4.
\end{equation}
Since the $2$-part of $\CL(-8p)$ is cyclic, $\rk_{16}\CL(-8p) = 1$ if and only if $\CL(-8p)$ has an element of order $16$. To get started, we first produce an element of order $4$ in $\CL(-8p)$ that we can write explicitly in terms of $u$ and $v$.

\subsubsection{A class of order $4$}
We now produce an ideal generating a class of order $4$ in the class group $\CL(-8p)$ when $p$ is a prime $\equiv -1\bmod 8$. This is the main ingredient in \cite{LW82}.
\\\\
When $n = p$ and $f = 1$, the ideal $\pt = \pt_f$ defined in Section \ref{IdealIdentity} is the prime ideal lying above $2$. If $\pt=(x+y\sqrt{-2p})$ for some $x, y\in \ZZ$, then $x^2+2py^2 = \Norm(\pt) = 2$, which is impossible. Hence the class of $\pt$ in $\CL(-8p)$ has order $2$.
\\\\
Now let $\uu$ be defined as in \eqref{defu} with $u$ and $v$ as above and $f = 1$. Lemma~\ref{LemmaIdealIdentity} shows that $\uu^2$ and $\pt$ are in the same ideal class in $\CL(-8p)$. Hence we have proved the following result.
\begin{lemma}\label{ord4}
Let $\uu$ be the ideal of $\ZZP$ defined as above. Then the ideal class of $\uu$ has order $4$ in $\CL(-8p)$.
\end{lemma}
\begin{remark}
Perhaps an easier, although more old-fashioned, way to prove Lemma~\ref{ord4} is via the theory of binary quadratic forms, as was done in \cite{LW82}. Let $[a, b, c]$ denote the $\SL$-equivalence class of the form $ax^2+bxy+cy^2$. The key observation is that $[u, -4v, 2u]$ has discriminant $16v^2-8u^2 = -8p$. To compose this class with itself, one can use the special case of the composition law for \textit{concordant forms}, which yields the class $[u, -4v, 2u]^2 = [u^2, -4v, 2] = [2, 0, p]$. The classes $[u, -4v, 2u]$ and $[2, 0, p]$ correspond to the ideal classes of $\uu$ and $\pt$, respectively.
\end{remark}

\subsubsection{Generating the $4$-Hilbert class field}
Let $p$ be a prime congruent to $-1\bmod 8$ and let $K = \QQP$. The $2$-Hilbert class field, also called the \textit{genus field} of $K$, is known to be $H_2 = K(\sqrt{2})$. Lemma~\ref{ord4} implies that $\rk_4\CL(-8p) = 1$, and our aim is to generate the $4$-Hilbert class field $H_4$ over $H_2$ by adjoining an element that we can write explicitly in terms of $u$ and $v$.
\\\\
Define $\pi\in\ZZT$ by setting $\pi = \nu$ with $\nu$ as in \eqref{nu}, i.e., $\pi = u+v\sqrt{2}$. The following proposition achieves our aim.
\begin{prop}\label{4HCF}
Let $K = \QQP$, and let $\pi$ be as above. Then the $4$-Hilbert class field of $K$ is
$$H_4 = H_2(\sqrt{\ve\pi}).$$
\end{prop}
\begin{proof}
Since the $2$-part of the class group $\CL(-8p)$ is cyclic, it suffices to show that $H_2(\sqrt{\ve\pi})$ is an unramified, cyclic, degree-$4$ extension of $K$.
\\\\
We apply the lemmas of Sections \ref{SpecialFields} and \ref{ArtinComp} with $n = m = p$, $e = 1$, and $u$ and $v$ as above. By Lemma~\ref{P1}, the extension $H_2(\sqrt{\ve\pi})/K$ is cyclic of degree $4$. By Lemma~\ref{P2}, $H_2(\sqrt{\ve\pi})/K$ is unramified over the prime ideal $\pp = (p, \sqrt{-2p})$ of $K$ lying over $p$. Finally, by part (1) of Lemma~\ref{P3}, $H_2(\sqrt{\ve\pi})/K$ is unramified over the prime ideal $\pt = (2, \sqrt{-2p})$ of $K$ lying over $2$.
\end{proof}

\subsubsection{Conclusion of the proof of Proposition~\ref{generalLW}}
By Lemma~\ref{ord4}, $\rk_{16}\CL(-8p) = 1$ if and only if the ideal class of $\uu$ belongs to $\CL(-8p)^4$. By Proposition~\ref{4HCF}, this is true if and only if the Artin symbol of $\uu$ in $H_4 = H_2(\sqrt{\ve\pi})$ is trivial. In the notation of Section \ref{ArtinComp}, we have that $H_2 = G$, so that $\rk_{16}\CL(-8p) = 1$ if and only if 
$$
\left(\frac{\uu}{G(\sqrt{\ve\pi})/K}\right) = \mathrm{Id}.
$$
By Lemma~\ref{keylemma2}, this occurs if and only if $\left(\frac{v}{u}\right)\chi(u) = 1$, which proves Proposition~\ref{generalLW}.

\subsection{Proof of Proposition~\ref{generalLW2}}
As in the statement of Proposition~\ref{generalLW2}, let $u_1$ and $v_1$ be integers such that $u_1$ is odd and positive and such that $u_1^2-2v_1^2>0$. We define $u_2$ and $v_2$ by the equality
\begin{equation}\label{defu2v2}
u_2+v_2\sqrt{2} = \varepsilon^8(u_1+v_1\sqrt{2}) =(577u_1+816v_1) + (408u_1+577v_1)\sqrt{2},
\end{equation}
where, as before, $\varepsilon = 1+\sqrt{2}$. Our goal is to prove the following equality of Jacobi symbols
\begin{equation}\label{prop2goal}
\left(\frac{v_1}{u_1}\right) = \left(\frac{v_2}{u_2}\right).
\end{equation}
By the Euclidean algorithm, we have the equality
$$
\gcd(u_1, v_1) = \gcd(u_2, v_2).$$
First, if $\gcd(u_1, v_1) = \gcd(u_2, v_2)> 1$, then both sides of \eqref{prop2goal} are equal to $0$, and hence \eqref{prop2goal} holds true.
\\\\
Now suppose $\gcd(u_1, v_1) = \gcd(u_2, v_2) = 1$. Let
$$
n = u_1^2-2v_1^2 = u_2^2-2v_2^2,
$$
and let $K = \QQ(\sqrt{-2n})$ as in Section \ref{SpecialFields}. 
Set $\tau = 4\sqrt{-2n}$. Let $\uu_1$ (resp. $\uu_2$) be the ideal of the imaginary quadratic order $\ZZ[\tau]$ (of discriminant $16\cdot -8n$) defined by \eqref{defu} with $(u, v) = (u_1, v_1)$ (resp. $(u, v) = (u_2, v_2)$) and $f = 4$. The ideals $\uu_1$ and $\uu_2$ satisfy the following key property.
\begin{lemma}\label{sameclass}
The ideals $\uu_1$ and $\uu_2$ belong to the same ideal class in the class group $\CL(16\cdot -8n)$ of the imaginary quadratic order $\ZZ[\tau]$.
\end{lemma}
\begin{proof}
Let $k\in\{1, 2\}$. By Lemma~\ref{LemmaIdealIdentity}, we have
$$
(8v_k-\tau) = \pt_{4, k}\uu_k^2
$$
where $\pt_{4, k} = (8v_k-\tau, 32)$ is as in Section \ref{IdealIdentity}. By \eqref{defu2v2}, we have
$$
8v_2 = 8(408u_1+577v_1) = 8v_1 + 32(102u_1+144v_1),
$$
so that
$$
\pt_{4, 2} = (8v_2-\tau, 32) = (8v_1 - \tau, 32) = \pt_{4, 1}.
$$
Therefore
\begin{equation}\label{compareu1u2}
\uu_2^2 = \frac{8v_2-\tau}{8v_1-\tau}\uu_1^2.
\end{equation}
Let $\alpha = (17u_1+24v_1)+3\tau$. We claim that
\begin{equation}\label{alphaclaim}
\left(\frac{\alpha}{u_1}\right)^2 = \frac{8v_2-\tau}{8v_1-\tau}.
\end{equation}
We first note that
\begin{equation}\label{fracexp}
\begin{array}{rcl}
\displaystyle{\frac{8v_2-\tau}{8v_1-\tau}} & = & \displaystyle{\frac{8v_2-\tau}{8v_1-\tau}\cdot \frac{8v_1+\tau}{8v_1+\tau}}\\
 & = & \displaystyle{\frac{64v_1v_2 + 32n + 8(v_2-v_1)\tau}{64v_1^2+32n}} \\
 & = & \displaystyle{\frac{64v_1(408u_1+577v_1) + 32n + 8(408u_1+576v_1)\tau}{32u_1^2}} \\
 & = & \displaystyle{\frac{n + 2v_1(408u_1+577v_1) + (102u_1+144v_1)\tau}{u_1^2}}.
\end{array} 
\end{equation}
Expanding $\alpha^2$, we get
\begin{equation}\label{alphasq}
\begin{array}{rcl}
\alpha^2 & = & 289u_1^2+576v_1^2+816u_1v_1-288n+(102u_1+144v_1)\tau \\
 & = & u_1^2+1152 v_1^2 + 816u_1v_1 +(102u_1+144v_1)\tau \\
 & = & n + 1154v_1^2 + 816u_1v_1 +(102u_1+144v_1)\tau \\
 & = & n + 2v_1(408u_1+577v_1) +(102u_1+144v_1)\tau.    
\end{array}
\end{equation}
Comparing the last line of \eqref{alphasq} with the numerator in the last line of \eqref{fracexp}, we obtain \eqref{alphaclaim}.
\\\\
Now \eqref{compareu1u2} and \eqref{alphaclaim} imply that
\begin{equation}\label{compareu1u22}
u_1^2\uu_2^2 = \alpha^2\uu_1^2.
\end{equation}
By \eqref{normu}, $\Norm(\uu_2) = u_2$. Hence $\Norm(\uu_2)$ is odd, and since $u_1$ is also odd, we find that $u_1^2\uu_2^2$ is coprime to the conductor $f = 4$ of $\ZZ[\tau]$, and hence factors uniquely into prime ideals. Therefore \eqref{compareu1u22} implies that $u_1\uu_2 = \alpha \uu_1$, which proves the lemma.
\end{proof}
\begin{remark}
There is a shorter proof of Lemma~\ref{sameclass} via the theory of binary quadratic forms. The $\SL$-equivalence classes of binary quadratic forms of discriminant $16\cdot -8n$ corresponding to the ideals $\uu_1$ and $\uu_2$ of $\ZZ[\tau]$ are $[u_1, 16v_1, 32u_1]$ and $[u_2, 16v_2, 32u_2]$, respectively. The matrix
$$\left( \begin{array}{cc}
17 & 96 \\
3 & 17 \end{array} \right) \in \SL$$
transforms the quadratic form $[u_1, 16v_1, 32u_1]$ into $[u_2, 16v_2, 32u_2]$, which proves the lemma.
\end{remark}
Now, for $k\in\{1, 2\}$, define $\nu_k = u_k+v_k\sqrt{2}$ similarly as in Section \ref{SpecialFields}. Then
\begin{equation}\label{pi1pi2}
\nu_2 = \varepsilon^8\nu_1.
\end{equation}
Since $\sqrt{2}$ is contained in $G = K(\sqrt{2})$, $\epsilon^8$ is a square in $G$. Hence the fields $G(\sqrt{\ve\nu_1})$ and $G(\sqrt{\ve\nu_2})$ are equal, and so we define
$$
L = G(\sqrt{\ve\nu_1}) = G(\sqrt{\ve\nu_2}).
$$
By Lemma~\ref{P4}, $L$ is contained in the ring class field of $\ZZ[\tau]$. Hence, by Lemma~\ref{sameclass}, the images of both $\uu_1$ and $\uu_2$ under the map \eqref{RCF1} coincide, i.e.,
$$
\left(\frac{\uu_1}{L/K}\right) = \left(\frac{\uu_2}{L/K}\right).
$$
Applying Lemma~\ref{keylemma2}, we get
$$
\left(\frac{v_1}{u_1}\right)\chi(u_1) = \left(\frac{v_2}{u_2}\right)\chi(u_2).
$$
Equation \eqref{defu2v2} implies that
\begin{equation}\label{upu16}
u_2 = 577u_1+816v_1 \equiv u_1\bmod 16.
\end{equation}
Hence, as $\chi$ is a character modulo $16$, we have $\chi(u_1) = \chi(u_2)$, and so Proposition~\ref{generalLW2} is finally proved.  

\section{Sums over primes}\label{Sumsoverprimes}
Above, we defined the governing symbol $\left\langle p \right\rangle$ for a prime $p\equiv -1\bmod 16$ in terms of particular integer solutions $u$ and $v$ to the equation $p = u^2-2v^2$. The main lemma that we will use to prove Theorem~\ref{mainThm}, i.e., that these governing symbols oscillate, is a proposition due to Friedlander, Iwaniec, Mazur and Rubin \cite{FIMR}. We now state this proposition in our context.

\subsection{A result of Friedlander, Iwaniec, Mazur, and Rubin}
Recall that an element $w = u+v\sqrt{2}\in\ZZT$ is \textit{totally positive} if and only if $\Norm(w) = u^2-2v^2>0$ \textit{and} $u>0$. We sometimes write $w\succ 0$ to say that $w$ is totally positive. 
\\\\
Since $\ZZT$ is a principal ideal domain and since the norm of the fundamental unit $\varepsilon$ over $\QQ$ is $-1$, an ideal $\nn$ in $\ZZT$ can always be generated by a totally positive element. For an ideal $\nn$ of $\ZZT$, recall that the norm of $\nn$ is given by
$$\Norm(\nn):= u^2-2v^2,$$
where $u+v\sqrt{2}$ is a totally positive generator of $\nn$.
\\\\
We now define an analogue of the von Mangoldt function $\Lambda$ for the ring $\ZZT$. For a non-zero ideal $\nn$ of $\ZZT$, we set
$$
\Lambda(\nn)=
\begin{cases}
\log(\Norm(\pp)) & \text{if }\nn = \pp^k\text{ for some prime ideal }\pp\text{ and integer }k\geq 1\\
0 & \text{otherwise.}
\end{cases}
$$
Hence $\Lambda$ is supported on powers of prime ideals.
\\\\
Given a sequence of complex numbers $\{a_{\nn}\}_{\nn}$ indexed by non-zero ideals in $\ZZT$, a good estimate for the sum of $a_{\nn}$ over prime ideals $\pp$ of norm $\Norm(\pp)\leq X$ can usually be derived from a good estimate of the ``smoother'' weighted sum
$$S(X):=\sum_{\Norm(\nn)\leq X}a_{\nn}\Lambda(\nn).$$
The idea in \cite{FIMR} (and even earlier in \cite{FI1}), is to bound $S(X)$ by combinations of linear and bilinear sums in $a_{\nn}$. Given a non-zero ideal $\dd$ of $\ZZT$, we define the linear sum
\begin{equation}\label{sumA}
A_{\dd}(X):= \sum_{\substack{\Norm(\nn)\leq X\\ \nn\equiv 0\bmod\dd}}a_{\nn}.
\end{equation}
Moreover, given two sequences of complex numbers $\{\alpha_{\mm}\}$ and $\{\beta_{\nn}\}$, each indexed by non-zero ideals in $\ZZT$, we define the bilinear sum
\begin{equation}\label{sumB}
B(M, N): = \sum_{\Norm(\mm)\leq M}\sum_{\Norm(\nn)\leq N}\alpha_{\mm}\beta_{\nn}a_{\mm\nn}.
\end{equation}
We consider bilinear sums where the complex numbers $\alpha_{\mm}$ and $\beta_{\nn}$ satisfy
\begin{equation}\label{ambn}
|\alpha_{\mm}|\leq \Lambda(\mm)\text{   and   }|\beta_{\nn}|\leq \tau(\nn),
\end{equation}  
where $\tau(\nn)$ denotes the number of ideals in $\ZZT$ dividing $\nn$. We now state \cite[Proposition 5.2, p.722]{FIMR} that we use to prove Theorem~\ref{mainThm}. 
\begin{prop}\label{sop}
Let $a_{\nn}$ be a sequence of complex numbers bounded by $1$ in absolute value and indexed by non-zero ideals of $\ZZT$. Suppose that there exist two real numbers $0< \theta_1, \theta_2<1$ such that: for every $\epsilon>0$, we have
\begin{equation}\tag{A}
A_{\dd}(X)\ll_{\epsilon} X^{1-\theta_1+\epsilon}
\end{equation}
uniformly for all non-zero ideals $\dd$ of $\ZZT$ and all $X\geq 2$, and
\begin{equation}\tag{B}
B(M, N)\ll_{\epsilon} (M+N)^{\theta_2}(MN)^{1-\theta_2+\epsilon}
\end{equation}
uniformly for all $M, N\geq 2$ and sequences of complex numbers $\{\alpha_{\mm}\}$ and $\{\beta_{\nn}\}$ satisfying \eqref{ambn}.
\\
Then for all $X\geq 2$ and all $\epsilon > 0$, we have the bound
$$S(X)\ll_{\epsilon}X^{1-\frac{\theta_1\theta_2}{2+\theta_2}+\epsilon}.$$  
\end{prop}
In other words, power-saving estimates for linear and bilinear sums imply power-saving estimates for sums supported on primes. Note that this result is now classical in the context of rational integers, thanks to the pioneering work of Vinogradov \cite{Vino}. 

\subsection{Extending governing symbols}
In light of Proposition~\ref{sop}, our current goal is to define a sequence $\{a_{\nn}\}$ indexed by non-zero ideals $\nn$ of $\ZZT$ so that if $p\equiv -1\bmod 16$ is a prime and $\pp$ is a prime ideal of $\ZZT$ lying above $p$, then $a_{\pp}$ coincides with the governing symbol $\left\langle p \right\rangle$ defined in \eqref{govp}. We first define a spin symbol $[\cdot]$ for all totally positive elements of $\ZZT$. We put
$$
[u+v\sqrt{2}]:=
\begin{cases}
\left(\frac{v}{u}\right) & \text{ if }u\text{ is odd} \\
0 & \text{ otherwise}
\end{cases}
$$
If $u+v\sqrt{2}\succ 0$ generates a prime ideal $\pp$ in $\ZZT$ lying above a prime $p\equiv -1\bmod 16$ and if $u\equiv 1\bmod 16$, then $[u+v\sqrt{2}] = \left\langle p \right\rangle$, by definition \eqref{govp}. Indeed, the condition $u\equiv 1\bmod 16$ implies that $\chi(u) = 1$ and also that $\left(\frac{-v}{u}\right) = \left(\frac{v}{u}\right)$ (note that one of $v$ and $-v$ is congruent to $1\bmod 4$). It is also convenient that exactly one of the four elements $\varepsilon^{2k}(u+v\sqrt{2}) = u_k+v_k\sqrt{2}$ ($0\leq k\leq 3$) satisfies $u_k\equiv 1\bmod 16$. Indeed, multiplying $u+v\sqrt{2}$ by $\varepsilon^2$ (resp. $\varepsilon^4$) transforms $(u, v)$ into $(3u+4v, 2u+3v)$ (resp. $(17u+24v, 12u+17v)$), and hence $u_2\equiv u_0+8\bmod 16$; one can now easily check that multiplying $u+v\sqrt{2}$ successively by $\varepsilon^2$ cycles $u\bmod 16$ through the set $\{1, 7, 9, 15\}$.
\\\\
Proposition~\ref{generalLW2} states that $[w] = [\varepsilon^8w]$ for any odd and totally positive $w\in\ZZT$, so, in light of the preceding discussion, we might naively define $a_{\nn} = \sum_{k = 0}^3[\varepsilon^{2k}w]$, where $w\succ 0$ is any totally positive generator of $\nn$. This definition does not quite suffice for our purposes because we want to isolate those $p$ that are congruent to $-1\bmod 16$ and representations $p = u^2-2v^2$ with $u\equiv 1\bmod 16$. Hence we weigh the expression above by Dirichlet characters modulo $16$. More precisely, for each pair of Dirichlet characters $\phi$ and $\psi$ modulo $16$ and totally positive $u+v\sqrt{2}$, we set
\begin{equation}\label{WGS}
[u+v\sqrt{2}]_{\phi,\psi} :=[u+v\sqrt{2}]\phi(-u^2+2v^2)\psi(u).
\end{equation}
For a non-zero ideal $\nn$ in $\ZZT$ generated by a totally positive element $w$, we set
\begin{equation}\label{Wann}
a_{\phi, \psi, \nn} := \sum_{k = 0}^3[\varepsilon^{2k} w]_{\phi, \psi}.
\end{equation}
This is still well-defined, i.e., independent of the choice of $w\succ 0$, by Proposition~\ref{generalLW2} and by \eqref{upu16}. We will apply Proposition~\ref{sop} to $8^2$ sequences $\{a_{\phi, \psi, \nn}\}_{\nn}$, one for each pair of Dirichlet characters $\phi$, $\psi$, and then add together the corresponding $8^2$ sums $S_{\phi, \psi}(X)$ to obtain Theorem~\ref{mainThm}. It is now easy to check  
\begin{lemma}\label{key}
If $p$ is a prime and $\pp$ is a prime ideal lying above $p$, then we have
$$\frac{1}{8^2}\sum_{\phi}\sum_{\psi}a_{\phi, \psi, \pp} = 
\begin{cases}
\left\langle p \right\rangle & \text{ if }p\equiv -1\bmod 16 \\
0 & \text{ otherwise}.
\end{cases}$$
\end{lemma}
Hence, to prove Theorem~\ref{mainThm}, it now suffices to prove
\begin{theorem}\label{mainThm2}
Let $a_{\phi, \psi, \nn}$ be defined as in \eqref{Wann}. For every $\epsilon>0$, there is a constant $C_{\epsilon}>0$ depending only on $\epsilon$ such that for every $X\geq 2$, we have
$$\left|\sum_{\Norm(\nn)\leq X}a_{\phi, \psi, \nn}\Lambda(\nn)\right| \leq C_{\epsilon}X^{\frac{149}{150}+\epsilon}.$$
\end{theorem}

\section{Fundamental domains}\label{Fundamentaldomains}
In order to obtain power-saving cancellation for linear and bilinear sums as in Proposition~\ref{sop}, we will have to choose generators of $\nn$ in \eqref{Wann} carefully. The problem reduces to constructing a convenient fundamental domain for the action of $\varepsilon^2 = 3+2\sqrt{2}$ on totally positive elements of $\ZZT$. Such constructions are standard (see for instance \cite[Chapter 6]{Marcus} or \cite[Section 4]{FIMR}). For the sake of completeness and explicitness, we give a simple argument tailored to our specific needs. Let 
\begin{equation}\label{defOmega}
\Omega := \left\{(u, v)\in \RR^2: u>0, -u<\sqrt{2}v<u \right\}.
\end{equation}
Then the lattice points $(u, v)\in \Omega\cap\ZZ^2$ precisely enumerate the totally positive elements $w = u + v\sqrt{2}$. The group $\left\langle \ve^2 \right\rangle$ of totally positive units of $\ZZT$ acts on the totally positive elements of $\ZZT$ by multiplication, and this induces an action $\left\langle \ve^2 \right\rangle\times\Omega\rightarrow \Omega$ given by
$$
a+b\sqrt{2} \cdot (u, v) := (au+2bv, bu+av).
$$
Let $\DD$ be the subset of $\Omega$ defined by
\begin{equation}\label{fundDom}
\DD := \left\{(u, v)\in \RR^2: u>0, -u<2v\leq u \right\}
\end{equation}
We claim that the region $\DD$ is a fundamental domain for the action of $\varepsilon^2$ on $\Omega$ in the following sense.
\begin{lemma}\label{lemFD}
For each element $(u, v)\in \Omega\cap\ZZ^2$, there exists exactly one integer $k$ such that $\ve^{2k}\cdot(u, v)\in\DD$.
\end{lemma}
\begin{proof}
Since $\ve^2 = 3+2\sqrt{2}>1$, we have that $\ve^{2k}>\ve^{2j}$ whenever $k>j$, that $\ve^{2k}\rightarrow 0$ as $k\rightarrow -\infty$, and that $\ve^{2k}\rightarrow \infty$ as $k\rightarrow \infty$. Moreover, given $(u, v)\in \Omega\cap\ZZ^2$, we have
$$
\frac{\ve^2(u+v\sqrt{2})}{\ve^{-2}(u-v\sqrt{2})} = \ve^4\cdot\frac{u+v\sqrt{2}}{u-v\sqrt{2}}.
$$
Hence, given $(u, v)\in \Omega\cap\ZZ^2$, there exists a unique integer $k$ such that
$$
\ve^{-2}<\frac{\ve^{2k}(u+v\sqrt{2})}{\ve^{-2k}(u-v\sqrt{2})}\leq \ve^{2}.
$$
The lemma follows upon noticing that for $(u, v)\in\Omega$, we have $(u, v)\in \DD$ if and only if $\ve^{-2}<(u+v\sqrt{2})/(u-v\sqrt{2})\leq \ve^{2}$.
\end{proof}
An immediate consequence of Lemma~\ref{lemFD} is the following proposition.
\begin{prop}\label{propFD}
Suppose that $\nn$ is a non-zero ideal of $\ZZT$. Then $\nn$ has a unique generator in $\DD$.
\end{prop}

\subsection{Geometry of numbers in the fundamental domain: the Lipschitz principle}\label{lipschitz}
We now briefly turn to the problem of counting lattice points and boxes inside certain compact subsets of the fundamental domain $\DD$. We state a lemma of Davenport (see \cite{Dav} and \cite{DavC}).
\\\\
Let $\RRR$ be a compact, Lebesgue measurable subset of $\RR^n$. Suppose that $\RRR$ satisfies the following two conditions:
\begin{enumerate}
	\item Any line parallel to one of the $n$ coordinate axes intersects $\RRR$ in a set of points which, if not empty, consists of at most $h$ intervals, and
	\item The same is true (with $m$ in place of $n$) for any of the $m$-dimensional
regions obtained by projecting $\RRR$ on one of the coordinate spaces defined by
equating a selection of $n-m$ of the coordinates to zero; and this condition is
satisfied for all $m$ from $1$ to $n-1$.
\end{enumerate}
\begin{lemma}[Davenport]\label{dav}
If $\RRR$ satisfies conditions (1) and (2) above, then
$$
\left|\RRR\cap\ZZ^n - \Vol(\RRR)\right|\leq \sum_{m = 0}^{n-1}h^{n-m}V_m
$$
where $V_m$ is the sum of the $m$-dimensional volumes of the projections of $\RRR$ on
the various coordinate spaces obtained by equating any $n-m$ coordinates to zero,
and $V_0=1$ by convention.
\end{lemma}
We will apply Lemma~\ref{dav} to the fundamental domain $\DD\subset\RR^2$ as well as certain variations thereof.
\\\\
Let $k\geq 0$ be an integer, and define
$$
\DD_k = \DD\cup\ve^2\cdot\DD\cdots\cup\ve^{2k}\cdot\DD.
$$
Let $X>0$. Then the region 
\begin{equation}\label{DX}
\DD_k(X) := \{(u, v)\in\DD_k: u^2-2v^2\leq X \}
\end{equation}
is a compact subset of $\RR^2$ and satisfies conditions (1) and (2) above with $h = 2$. Moreover, one can check that there exist positive real numbers $a_k$ and $\ell_k$ such that
\begin{equation}\label{defao}
\Vol(\DD_k(X)) = a_kX\ \ \ \ \ \ \ \text{ and }\ \ \ \ \ \ \ \ \Vol(\partial(\DD_k(X))) = \ell_kX^{\frac{1}{2}}.
\end{equation}
Now let $L:\RR^2\rightarrow\RR^2$ be an invertible linear transformation of the form  
$$
L\left( \begin{array}{c}
x\\
y\end{array} \right):=
\left( \begin{array}{cc}
a & b\\
c & d\end{array} \right)
\left( \begin{array}{c}
x\\
y\end{array} \right)+
\left( \begin{array}{c}
x_0\\
y_0\end{array} \right),
$$
of determinant $D:=ad-bc\neq 0$. Then $L(\DD_k(X))$ is a compact subset of $\RR^2$ that also satisfies conditions (1) and (2) above, also with $h = 2$. We define the diameter of $L$ to be $\diam(L) := |a|+|b|+|c|+|d|$. Then
$$
\Vol(L(\DD_k(X))) = |D|\Vol(\DD_k(X))
$$
and
$$
\Vol(\partial(L(\DD_k(X)))) = O(\diam(L)\cdot X^{\frac{1}{2}}),
$$
where the implied constant is absolute.

\section{Linear sums}\label{Linearsums}
In this section we prove that the estimate (A) from Proposition~\ref{sop} holds for the sequence $\{a_{\phi, \psi, \nn}\}_{\nn}$ defined in \eqref{Wann} with $\theta_1 = 1/6$.
\begin{prop}\label{propA}
Let $a_{\nn} = a_{\phi, \psi, \nn}$, where $a_{\phi, \psi, \nn}$ is defined as in \eqref{Wann}, and let $A_{\dd}(X)$ be defined as in \eqref{sumA}. Then for all $\epsilon>0$ and all $X\geq 2$, we have
$$A_{\dd}(X)\ll_{\epsilon}X^{\frac{5}{6}+\epsilon}.$$ 
\end{prop}
\begin{proof}
Recall that
$$A_{\dd}(X) = \sum_{\substack{\Norm(\nn)\leq X\\ \nn\equiv 0\bmod\dd}}a_{\nn}.$$
Since the sequence $a_{\nn}$ is supported on odd ideals $\nn$, we see that $A_{\dd}(X) = 0$ unless $\dd$ is odd. Hence we may assume without loss of generality that $\dd$ is an odd ideal. Let 
\begin{equation}\label{reg}
\RRR(X) := \DD_4(X) = \left\{(u, v)\in \DD\cup\varepsilon^2\DD\cup\varepsilon^4\DD\cup\varepsilon^6\DD: u^2-2v^2\leq X\right\}.
\end{equation}
By Proposition~\ref{propFD} and definition \eqref{Wann}, we have
$$
A_{\dd}(X) = \sum_{\substack{(u, v)\in\RRR(X)\\ u+v\sqrt{2}\equiv 0\bmod \dd}}[u+v\sqrt{2}]_{\phi, \psi},
$$
where $[u+v\sqrt{2}]_{\phi, \psi}$ is defined as in \eqref{WGS}.
\\\\
We now reformulate the congruence condition $u+v\sqrt{2}\equiv 0\bmod \dd$. Proposition~\ref{propFD} implies that there is an element $d_1+d_2\sqrt{2}\in\DD$ which generates $\dd$. Then the congruence above is equivalent to saying that there exist integers $e_1$ and $e_2$ such that $u+v\sqrt{2} = (d_1+d_2\sqrt{2})(e_1+e_2\sqrt{2})$, i.e., such that
$$u  = d_1e_1 + 2d_2e_2$$
and
$$v = d_2e_1 + d_1e_2.$$
In other words, $(u, v)$ is in the image of the linear transformation 
$$
L_{\dd} :=
\left( \begin{array}{cc}
d_1 & 2d_2\\
d_2 & d_1\end{array} \right): \ZZ^2\rightarrow \ZZ^2
$$
of determinant $D:=\Norm(\dd) = d_1^2-2d_2^2$. Hence we define
$$
\RRR(\dd, X) := \left\{(u, v)\in \RRR(X): (u, v)\in\text{Image}(L_{\dd})\right\},
$$
\begin{center}
\begin{figure}[h]
\includegraphics[scale = 0.55]{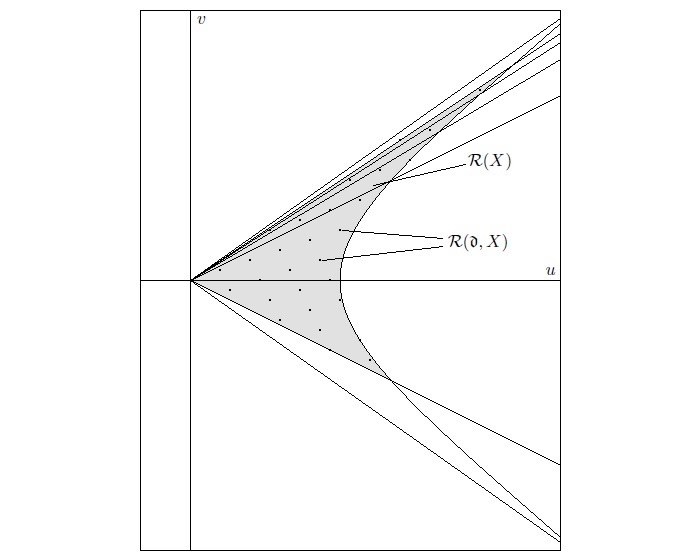}
\caption{The region $\RRR(X)$ and the lattice points $\RRR(\dd, X)$}
\label{fig:2}
\end{figure}
\end{center}
and we rewrite the sum $A_{\dd}(X)$ as
$$
A_{\dd}(X) = \sum_{(u, v)\in\RRR(\dd, X)}[u+v\sqrt{2}]_{\phi, \psi}.
$$
Using the fact that $|[u+v\sqrt{2}]_{\phi, \psi}|\leq 1$, we obtain the trivial bound 
\begin{equation}\label{boundTriv0}
\left|A_{\dd}(X)\right| \leq \sum_{(u, v)\in\RRR(\dd, X)} 1 = \sum_{L_{\dd}^{-1}\RRR(X)\cap\ZZ^2}1.
\end{equation}
Since $d_1+d_2\sqrt{2}\in\DD$, we have the inequalities
$$
\frac{d_1^2}{2}\leq D \leq d_1^2,
$$
which implies that $\diam(L_{\dd}^{-1}) \ll D^{-1/2}$. Hence Lemma~\ref{dav} gives
\begin{equation}\label{boundTriv1}
\left|A_{\dd}(X)\right| \leq a_4XD^{-1} + O(D^{-\frac{1}{2}}X^{\frac{1}{2}}+1) \ll XD^{-1} + X^{\frac{1}{2}}D^{-\frac{1}{2}} + 1,
\end{equation}
where the implied constant is absolute. This estimate will be useful when $D$ is large compared to $X$.
\\\\
Next we split the sum $A_{\dd}(X)$ into $8\cdot 16$ sums where the congruence classes of $u$ and $v$ modulo $16$ are fixed, say $u\equiv u_0\bmod 16$ and $v\equiv v_0\bmod 16$ for some congruence classes $u_0$ and $v_0$ modulo $16$ with $u_0$ invertible modulo $16$. For $u$ and $v$ satisfying these congruences, we have 
$$[u+v\sqrt{2}]_{\phi, \psi} = \delta(u_0, v_0)\left(\frac{v}{u}\right),$$
where $\delta(u_0, v_0)\in\{\pm 1\}$ depends only on the congruence classes $u_0$ and $v_0$ modulo $16$. Hence it remains to give estimates for sums of the type
$$A_{\dd}(u_0, v_0, X) := \sum_{(u, v)\in\RRR(u_0, v_0, \dd, X)}\left(\frac{v}{u}\right),$$
where
$$\RRR(u_0, v_0, \dd, X) := \left\{(u, v)\in \RRR(\dd, X): (u, v)\equiv(u_0, v_0)\bmod 16\right\}.$$
Splitting the sum according to the value of $u$, we obtain 
\begin{equation}\label{intermed1}
A_{\dd}(u_0, v_0, X) = \sum_{\substack{0\leq u\leq R_1(X)\\ u\equiv u_0\bmod 16}}A_{u, \dd}(v_0, X),
\end{equation}
where
$$A_{u, \dd}(v_0, X) := \sum_{\substack{v\in I_u\\ (u, v)\in L_{\dd}(\ZZ^2)  \\v\equiv v_0\bmod 16}}\left(\frac{v}{u}\right).$$
Here
$$
R_1(X) = \sup\{u\in\RR:\ (u, v)\in\RRR(X)\} \ll X^{\frac{1}{2}}
$$
and $I_u$ is an interval (or a union of 2 disjoint intervals) of size $\leq 2R_2(X)$, where
$$
R_2(X) = \sup\{|v|\in\RR:\ (u, v)\in\RRR(X)\} \ll X^{\frac{1}{2}}.
$$
We now unwind the condition $(u, v)\in L_{\dd}(\ZZ^2)$, i.e., that $(u, v)$ is in the image of $L_{\dd}$. Consider the system of equations in $x$ and $y$:
\begin{equation}\label{sysxy}
\begin{cases}
u = d_1x+2d_2y \\
v = d_2x+d_1y.
\end{cases}
\end{equation}
Let $d := \gcd(d_1, d_2)$ and write $d_1 = dd_1'$, $d_2 = dd_2'$. Recall that $\dd$ and so also $d_1$ is odd, so that $d = \gcd(d_1, 2d_2)$. If the system \eqref{sysxy} has a solution over $\ZZ$, then $d$ must divide $u$. This means that
$$A_{\dd}(u_0, v_0, X) = \sum_{\substack{0\leq u\leq R_1(X)\\ u\equiv u_0\bmod 16\\ u\equiv 0\bmod d}}A_{u, \dd}(v_0, X).$$
Now suppose $u\equiv 0\bmod d$, and let $x_u, y_u\in\ZZ$ be such that 
$$u = d_1x_u+2d_2y_u.$$
Then all solutions $(x, y)\in\ZZ^2$ to the first equation in \eqref{sysxy} are given by 
$$(x, y) = (x_u-2d_2'k, y_u+d_1'k), \ \ \ k\in\ZZ.$$ 
Hence 
$$v = d_2\left(x_u-2d_2'k\right)+d_1\left(y_u+d_1'k\right) = d_2x_u+d_1y_u + Dk/d,$$
which means that \eqref{sysxy} has a solution over $\ZZ$ if and only if
$$v \equiv d_2x_u+d_1y_u\bmod D/d.$$
Note that $D$ is odd, so that $D/d$ and $16$ are coprime. Let $v_u$ be the congruence class modulo $16D/d$ such that
$$
\begin{cases}
v_u \equiv d_2x_u+d_1y_u \bmod D/d \\
v_u \equiv v_0 \bmod 16.
\end{cases}
$$
Thus we have proved that if $u\equiv 0\bmod d$, then
$$
A_{u, \dd}(v_0, X) = \sum_{\substack{v\in I_u\\ v\equiv v_u\bmod 16D/d}}\left(\frac{v}{u}\right).
$$ 
Let $e_u = \gcd(v_u, 16D/d)$, write $16D/d = e_ud_u$, $v_u = e_uv_u'$, and perform a change of variables $v = e_uv'$, so that
$$
A_{u, \dd}(v_0, X) =  \left(\frac{e_u}{u}\right)\sum_{\substack{v'\in I_u'\\ v'\equiv v_u'\bmod d_u}}\left(\frac{v'}{u}\right),
$$
where $I_u' = I_u/e_u$. Since $\text{gcd}(v_u', d_u) = 1$, we can now detect the congruence condition $v'\equiv v_u'\bmod d_u$ via Dirichlet characters modulo $d_u$. In other words,
\begin{equation}\label{intermed2}
A_{u, \dd}(v_0, X) =  \frac{1}{\varphi(d_u)}\left(\frac{e_u}{u}\right)\chi(\overline{v_u'})\sum_{\chi\bmod d_u}\sum_{v'\in I_u'}\chi(v')\left(\frac{v'}{u}\right),
\end{equation}
where $\overline{v_u'}$ denotes the multiplicative inverse of $v_u'$ modulo $d_u$. Let $\chi$ be a Dirichlet character modulo $d_u$. If the character
$$v'\mapsto \chi(v')\left(\frac{v'}{u}\right)$$
is trivial, then $u = fg^2$ for some $f$ dividing $d_u$ (and therefore dividing $16D/d$) and some integer $g$. The number of such $u\leq R_1(X)$ is 
$$\leq \tau(16D/d)R_1(X)^{\frac{1}{2}}\ll_{\epsilon} D^{\epsilon}X^{\frac{1}{4}}.$$
In this case we use the trivial bound
$$\sum_{v'\in I_u'}\chi(v')\left(\frac{v'}{u}\right)\ll \#I_u'\leq \#I_u \ll X^{\frac{1}{2}},$$
where the implied constant in $\ll$ is absolute. Hence the contribution of such $u$ to $A_{\dd}(u_0, v_0, X)$ is  
\begin{equation}\label{boundTriv2}
\ll_{\epsilon} D^{\epsilon}X^{\frac{3}{4}}.
\end{equation}
On the other hand, if the character $v'\mapsto \chi(v')\left(\frac{v'}{u}\right)$ is not trivial, its conductor is at most $16Du/d \ll DX^{\frac{1}{2}}$, and so the Polya-Vinogradov inequality gives the estimate
$$\sum_{v'\in I_u'}\chi(v')\left(\frac{v'}{u}\right)\ll_{\epsilon}D^{\frac{1}{2}}X^{\frac{1}{4}+\epsilon}.$$
Combining this with \eqref{intermed1}, \eqref{intermed2}, and \eqref{boundTriv2}, we have proved the bound
\begin{equation}\label{boundPV}
A_{\dd}(X) \ll_{\epsilon}D^{\frac{1}{2}}X^{\frac{3}{4}+\epsilon}.
\end{equation}
We use \eqref{boundPV} for $D < X^{1/6}$ and \eqref{boundTriv1} for $D\geq X^{1/6}$ to obtain
$$A_{\dd}(X)\ll_{\epsilon}X^{\frac{5}{6}+\epsilon}.$$
\end{proof}

\section{Bilinear sums}\label{Bilinearsums}
We are left with proving the estimate (B) from Proposition~\ref{sop}, which we do with $\theta_2 = 1/12$ in much the same way as in \cite[Sections 19-21, p.\ 1018-1028]{FI1}.
\begin{prop}\label{propB}
Let $a_{\nn} = a_{\phi, \psi, \nn}$, where $a_{\phi, \psi, \nn}$ is defined as in \eqref{Wann}, and let $B(M, N)$ be defined as in \eqref{sumB}. Then for all $\epsilon>0$ and all $M, N\geq 2$, we have
$$B(M, N)\ll_{\epsilon}\left(M+N \right)^{\frac{1}{12}}\left(MN\right)^{\frac{11}{12}+\epsilon}.$$
\end{prop}
Our basic strategy will be to prove a factorization formula of the type $[wz] = [w][z]\gamma(w, z)$, where $\gamma(w, z)$ is a quantity which oscillates in both arguments $w, z\in\ZZT$. We first develop some background necessary to define $\gamma(w, z)$ and then prove power-saving cancellation for general bilinear sums of the type $\sumsum_{w, z}\alpha_w\beta_z\gamma(w, z)$.

\subsection{Primitivity}
We say that an ideal $\aaa$ in $\ZZT$ is \textit{primitive} if whenever $\pp$ is a prime ideal dividing $\aaa$, then $\pp$ is unramified, of residue degree one, and $\opp$ does not divide $\aaa$. Here and after, if $x$ is an element or an ideal in $\ZZT$ we will use $\overline{x}$ to denote the conjugate of $x$ over $\QQ$. The main property of primitive ideals that we will use is that the inclusion $\ZZ\hookrightarrow \ZZT$ induces an isomorphism
\begin{equation}\label{primitive1}
\ZZ/(\Norm(\aaa))\stackrel{\sim}{\rightarrow} \ZZT/\aaa. 
\end{equation}
We call an ideal $\aaa$ (resp.\ element $w$) in $\ZZT$ \textit{odd} if $\Norm(\aaa)$ (resp.\ $\Norm(w)$) is an odd integer. An ideal in $\ZZT$ is odd if and only if every prime ideal that divides $\aaa$ is unramified. Hence, an ideal $\aaa$ is primitive if and only if $\aaa$ is odd and there is no rational prime $p$ dividing $\aaa$ (i.e., no rational prime $p$ such that $(p)$ divides $\aaa$). 
\begin{remark}
For instance, $\Norm(7) = 49$, but $\ZZT/(7)\cong \ZZT/(3+\sqrt{2})\times \ZZT/(3-\sqrt{2})\cong \ZZ/(7)\times\ZZ/(7) \ncong \ZZ/(49)$.
\end{remark}
For every integer $n$ we have the equality of quadratic residue symbols
\begin{equation}\label{primitive2}
\left(\frac{n}{\Norm(\aaa)}\right) = \left(\frac{n}{\aaa}\right),
\end{equation}
where the symbol on the left is the usual Jacobi symbol while the symbol on the right is the quadratic residue symbol in $\ZZT$, i.e., for $\alpha\in\ZZT$,
$$
\left(\frac{\alpha}{\aaa}\right) := \prod_{\pp^{k_{\pp}}\|\aaa}\left(\frac{\alpha}{\pp}\right)^{k_{\pp}},
$$
where 
$$
\left(\frac{\alpha}{\pp}\right) := 
\begin{cases}
1 & \text{if }(\alpha, \pp) = 1\text{ and }\alpha\equiv \square\bmod \aaa \\
-1 & \text{if }(\alpha, \pp) = 1\text{ and }\alpha\not\equiv \square\bmod \aaa \\
0 & \text{otherwise.} 
\end{cases}
$$
Now it follows immediately from \eqref{primitive1} and \eqref{primitive2} that 
\begin{equation}\label{primitive3}
\sum_{z\in \ZZT/\aaa}\left(\frac{z}{\aaa}\right) = \sum_{n\in \ZZ/(\Norm(\aaa))}\left(\frac{n}{\Norm(\aaa)}\right).
\end{equation}
The following is yet another characterization of primitive ideals.
\begin{lemma}\label{primitive4}
Suppose $\aaa\subset\ZZT$ is an odd ideal. Then $\aaa$ is primitive if and only if $\gcd(\aaa, \oaaa) = (1)$.
\end{lemma}
\begin{proof}
If $\aaa$ is not primitive, then there is a rational prime $p$ dividing $\aaa$. As $p$ is rational, it also divides $\oaaa$, and so $\gcd(\aaa, \oaaa)\neq (1)$. Conversely, if $\gcd(\aaa, \oaaa)\neq (1)$, then there is a prime ideal $\pp$ in $\ZZT$ such that both $\pp$ and $\opp$ divide $\aaa$. If $\pp$ is a prime of degree $2$, then $\pp = (p)$ for some rational prime $p$ and automatically $\aaa$ is not primitive. Otherwise, as $\aaa$ is odd and the only prime that ramifies in $\QQT/\QQ$ is $2$, we conclude that $\pp$ and $\opp$ are coprime, and hence that $\pp\opp$ divides $\aaa$. Once again, as $\pp\opp = (p)$ for a rational prime $p$, $\aaa$ is not primitive.
\end{proof}
Suppose $\aaa$ and $\bb$ are ideals in $\ZZT$. If one of $\aaa$ and $\bb$ is not primitive, then clearly their product $\aaa\bb$ is not primitive. Even if both $\aaa$ and $\bb$ are primitive, the product $\aaa\bb$ need \textit{not} be primitive. Nonetheless, we have the following lemma.
\begin{lemma}\label{primitive5}
Suppose $\aaa$ and $\bb$ are primitive. Let $\rrr = \gcd(\aaa, \obb)$ and $r = \Norm(\rrr)$. Then $\aaa\bb/(r)$ is primitive. In particular, $\aaa\bb$ is primitive if and only if $\gcd(\aaa, \obb) = (1)$.
\end{lemma}
\begin{proof}
Suppose $p$ is a rational prime such that $p$ divides $\aaa\bb$. Then $(p)$ cannot be a prime in $\ZZT$, because otherwise either $\aaa$ or $\bb$ is not primitive. Hence there exists a prime ideal $\pp\subset\ZZT$ such that $p = \pp\opp$ and $\pp|\aaa$. If $p^k$ is the exact power of $p$ dividing $\aaa\bb$, then the assumption that $\aaa$ and $\bb$ are primitive implies that $\pp^k|\aaa$ and $\opp^k|\bb$, which is true if and only if $\pp^k|\rrr$.
\end{proof}
There is another way to obtain a primitive ideal from a product of two odd primitive ideals $\aaa$ and $\bb$. We can write
$$
\aaa = \prod_{p\text{ split}}\pp^{a_p}\opp^{\oa_p}\ \ \ \ \ \ \ \text{and}\ \ \ \ \ \ \ \bb = \prod_{p\text{ split}}\pp^{b_p}\opp^{\ob_p}, 
$$
where $a_p\oa_p = b_p\ob_p = 0$ for every $p$. Let $\rrr = \gcd(\aaa, \bb)$ and let $r = \Norm(\rrr)$. If a prime $p$ divides $r$, after possibly interchanging the roles of $\pp$ and $\opp$ in the products above, we can assume that $\pp$ divides $\rrr$. For every such prime $p$, define
$$
\ccc_{a, p} = 
\begin{cases}
\pp^{a_p}&\text{ if }a_p\leq \ob_p, \\
1&\text{ otherwise},
\end{cases}
\ \ \ \ \ \ \text{and}\ \ \ \ \ \ 
\ccc_{b, p} = 
\begin{cases}
1&\text{ if }a_p\leq \ob_p, \\
\opp^{\ob_p}&\text{ otherwise},
\end{cases}
$$
and set
$$
\ccc_{a} = \prod_p\ccc_{a, p},\ \ \ \ \ \ \ \ccc_{b} = \prod_p\ccc_{b, p},\ \ \ \ \ \ \text{and}\ \ \ \ \ \ \ \ccc = \ccc_{a}\ccc_{b}.
$$
Then clearly
$$
\Norm(\ccc) = \Norm(\rrr) = r.
$$
Moreover, by construction 
$$
\gcd\left(\frac{\aaa}{\ccc_{a}}, \frac{\bb}{\overline{\ccc_{b}}}\right) = (1),
$$
so by Lemma~\ref{primitive4}, we conclude $\aaa\bb/\ccc$ is primitive. By construction, $\ccc$ is also primitive and coprime to $\aaa\bb/\ccc$. Therefore, using the Chinese Remainder Theorem and applying \eqref{primitive1} twice, we conclude that
\begin{equation}\label{primitive6}
\ZZT/\aaa\bb \cong \ZZT/(\aaa\bb/\ccc)\times\ZZT/\ccc \cong \ZZ/(W/r)\times \ZZ/(r),
\end{equation}
where $W = \Norm(\aaa\bb)$.
\\\\
Finally, we say that an element $w\in\ZZT$ is \textit{primitive} if and only if the principal ideal generated by $w$ is primitive. An equivalent definition is that $w = a+b\Sd$ is odd and $\gcd(a, b) = 1$. 

\subsection{A quasi-bilinear symbol with a reciprocity law}
For $w, z\in\ZZT$ with $w$ odd, we define the \textit{generalized Dirichlet symbol} $\gamma(w, z)$ to be
\begin{equation}\label{gammawz}
\gamma(w, z) := \left(\frac{\ow\oz}{(w)}\right),
\end{equation}
where $\left(\frac{\cdot}{\cdot}\right)$ is the quadratic residue symbol in $\QQT$. Our choice of terminology is inspired by the Dirichlet symbol defined in a slightly different context in \cite[Section 19, p.\ 1018-1021]{FI1}.
\\\\
The symbol $\gamma(w, z)$ factors as 
\begin{equation}\label{factorgamma}
\gamma(w, z) = \mul(w)\left(\frac{z}{(\ow)}\right),
\end{equation}
where, for odd $w\in\ZZT$, we define
\begin{equation}\label{defmul}
\mul(w) := \gamma(w, 1) = \left(\frac{\ow}{(w)}\right).
\end{equation}
By Lemma~\ref{primitive2}, if $w\in\ZZT$ is odd, then
$$
\mul(w) \neq 0 \Longleftrightarrow \gcd((w), (\ow)) = (1) \Longleftrightarrow w\text{ is primitive}.
$$
Hence the factor $\mul(w)$ restricts the support of $\gamma(w, z)$ to $w$ that are primitive. Furthermore, if $w$ is primitive, then
$\gcd((w), (\ow\oz)) = \gcd((w), (\oz))$, and so in this case $\gamma(w, z) = 0$ if and only if $\gcd((w), (\oz))\neq (1)$.
\\\\
The factor $\left(\frac{z}{(\ow)}\right)$ is completely multiplicative in $z$, so it follows from \eqref{factorgamma} that
\begin{equation}\label{gmult}
\gamma(w, z_1)\gamma(w, z_2) = \gamma(w, z_1z_2)\mul(w),
\end{equation}
for any $w$, $z_1$, and $z_2$ in $\ZZT$ such that $w$ is odd. Hence the symbol $\gamma(w, z)$ is multiplicative in $z$ except for a twist by $\mul(w)$. 
\\\\
The symbol $\gamma(w, z)$ also satisfies a reciprocity law, which is an important ingredient in our proof of Proposition~\ref{propB}. 
\begin{lemma}\label{reciprocitylem}
Let $w, z\in \ZZT$ such that both $w$ and $z$ are odd. Then
$$
\gamma(w, z)\gamma(z, w) = \mul(wz).
$$
\end{lemma}
\begin{proof}
We have
$$
\gamma(w, z)\gamma(z, w) = \left(\frac{\ow\oz}{(w)}\right)\left(\frac{\oz\ow}{(z)}\right) = \left(\frac{\ow\oz}{(wz)}\right) = \mul(wz).
$$
\end{proof}
Finally, we note that $\gamma(w, z)$ is periodic in the second argument. In fact, $\gamma(w, z_1) = \gamma(w, z_2)$ whenever $z_1\equiv z_2\bmod (\ow)$. In other words, $\gamma(w, \cdot)$ is a function on $\ZZT/(\ow)$. This allows us to prove the following analogue of \cite[Lemma 21.1, p.\ 1025]{FI1}, which will provide all of the cancellation that we need for Proposition~\ref{propB}.
\begin{lemma}\label{keycancellation}
Let $w_1, w_2\in\ZZT$ be primitive. Let $\rrr = \gcd((w_1), (\ow_2))$, $r = \Norm(r)$, $W = \Norm(w_1w_2)$. Then
$$
\left|\sum_{\substack{z\in \ZZT/(W)}}\gamma(w_1, z)\gamma(w_2, z) \right| =
\begin{cases}
W\varphi(r)\varphi(W/r) & \text{if }W\text{ and }r\text{ are squares} \\
0 & \text{otherwise.}
\end{cases}
$$
\end{lemma}
\begin{proof}
By \eqref{factorgamma}, we have 
$$
\gamma(w_1, z)\gamma(w_2, z) = \mul(w_1)\mul(w_2)\left(\frac{\oz}{(w_1w_2)}\right),
$$
and, as $w_1$ and $w_2$ are odd and primitive, $\mul(w_1)\mul(w_2)\neq 0$. Hence
$$
\left|\sum_{\substack{z\in \ZZT/(W)}}\gamma(w_1, z)\gamma(w_2, z) \right| = \left|\sum_{\substack{z\in \ZZT/(W)}}\left(\frac{\oz}{(w_1w_2)}\right) \right|.
$$
Now, as $W$ is rational, the map $z\mapsto \oz$ is an automorphism of the group $\ZZT/(W)$. Thus, we obtain
$$
\sum_{\substack{z\in \ZZT/(W)}}\left(\frac{\oz}{(w_1w_2)}\right) = \sum_{\substack{z\in \ZZT/(W)}}\left(\frac{z}{(w_1w_2)}\right) .
$$
As $\left(\frac{\cdot}{(w_1w_2)}\right)$ is already a function on $\ZZT/(w_1w_2)$, and 
$$
\#\left(\ZZT/(W)\right) = W\cdot \#\left(\ZZT/(w_1w_2)\right),
$$
we have
$$
\sum_{\substack{z\in \ZZT/(W)}}\left(\frac{z}{(w_1w_2)}\right) = W\sum_{\substack{z\in \ZZT/(w_1w_2)}}\left(\frac{z}{(w_1w_2)}\right).
$$
By \eqref{primitive6}, we have
$$
\ZZT/(w_1w_2) \cong \ZZT/(\alpha)\times \ZZT/(\beta),
$$
where $(\alpha)$ and $(\beta)$ are coprime primitive ideals of norm $W/r$ and $r$, respectively, satisfying $(w_1w_2) = (\alpha\beta)$. Hence
$$
\sum_{\substack{z\in \ZZT/(w_1w_2)}}\left(\frac{z}{(w_1w_2)}\right) = \sumsum_{\substack{z_{01}\bmod \alpha \\ z_{02}\bmod \beta}}\left(\frac{z}{(\alpha\beta)}\right),
$$
where
$$
z = z_{01}\cdot\beta\cdot\beta' + z_{02}\cdot\alpha\cdot\alpha'
$$
and $\alpha'$ and $\beta'$ are some elements of $\ZZT$ such that $\alpha\alpha'\equiv 1\bmod \beta$ and $\beta\beta'\equiv 1\bmod \alpha$. With these choices, we have
$$
\left(\frac{z}{(\alpha\beta)}\right) = \left(\frac{z}{(\alpha)}\right)\left(\frac{z}{(\beta)}\right) = \left(\frac{z_{01}}{(\alpha)}\right)\left(\frac{z_{02}}{(\beta)}\right). 
$$ 
Then, by \eqref{primitive3}, we have 
$$
\sum_{z_{01}\bmod \alpha}\left(\frac{z_{01}}{(\alpha)}\right) \sum_{z_{02}\bmod \beta}\left(\frac{z_{02}}{(\beta)}\right) = \sum_{a\in \ZZ/(W/r)}\left(\frac{a}{W/r}\right)\sum_{b\in \ZZ/(r)}\left(\frac{b}{r}\right),
$$
where the symbols on the right-hand side of the equality are the usual Jacobi symbols. For any positive integer $n$, we have
$$
\sum_{a\in \ZZ/(n)}\left(\frac{a}{n}\right) = 
\begin{cases}
\varphi(n)&\text{ if }n\text{ is a square}, \\
0&\text{ otherwise}.
\end{cases}
$$
Combining all of the equations above, we conclude the proof of the proposition. 
\end{proof}
We conclude this section by expressing $\gamma(w, z)$ as a Jacobi symbol. Suppose $w = a+b\sqrt{2}$ and $z = c+d\sqrt{2}$, with $w$ primitive and totally positive. Then
$$
\left(\frac{\ow\oz}{(w)}\right) = \left(\frac{wz+\ow\oz}{(w)}\right) = \left(\frac{2ac+4bd}{a^2-2b^2}\right).
$$
Moreover, as $w$ is primitive, every prime factor of $\Norm(w) = a^2-2b^2$ is congruent to $\pm 1$ modulo $8$, so $\left(\frac{2}{a^2-2b^2}\right) = 1$. Hence
\begin{equation}\label{gammaJacobi}
\gamma(w, z) = \left(\frac{ac+2bd}{a^2-2b^2}\right).
\end{equation} 

\subsection{Double oscillation of $\gamma(w, z)$}
We can now prove some general bilinear sum estimates that we will use to deduce Proposition~\ref{propB}. Let $\alpha = \{\alpha_{w}\}$ and $\beta = \{\beta_{z}\}$ be two sequences of complex numbers, each indexed by non-zero elements in $\ZZT$, such that
\begin{equation}\label{seqCond}
|\alpha_w|\leq \log(\Norm(w))\tau(w)\text{  and  }|\beta_z|\leq \log(\Norm(z))\tau(z)
\end{equation}
for all $w$ and $z$ in $\ZZT$. For a positive real number $X$, let $\DD(X) = \DD_0(X)$ as in \eqref{DX}. Set
$$
C := \limsup_{X\rightarrow \infty}\{u:\ (u, v)\in\DD(X)\}\cdot X^{-\frac{1}{2}}.
$$
and note that $C<\infty$. Next, for a positive real number $X$, we define the ``cone''
$$
\BB(X):=\{(u, v)\in\Omega:\ 0<u\leq CX^{\frac{1}{2}}\},
$$
where $\Omega$ is the region, defined in \eqref{defOmega}, which enumerates the totally positive elements in $\ZZT$. Hence the set of elements in $\ZZT$ enumerated by $\cup_{X>0}\BB(X) = \Omega$ is closed under multiplication. Note also that $\DD(X)\subset \BB(X)$ for every real number $X$. For a subset $\mathcal{S}$ of $\RR^2$ and an element $u+v\sqrt{2}\in\ZZT$, we will say that $u+v\sqrt{2}\in\mathcal{S}$ to mean that $(u, v)\in\mathcal{S}\cap\ZZ^2$. Finally, for positive real numbers $M$ and $N$, we define the bilinear sum
\begin{equation}\label{sumQ}
Q(M, N; \alpha, \beta): = \suma_{\substack{w \in\DD(M)}}\sum_{\substack{z \in \BB(N)}}\alpha_{w}\beta_{z}\gamma(w, z),
\end{equation}
where $\suma_w$ restricts the summation to primitive $w$. The first result we prove is a standard consequence of the Cauchy-Schwartz inequality and Lemma~\ref{keycancellation}.
\begin{lemma}\label{generalQsum1}
For every $\epsilon > 0$, there is a constant $C_{\epsilon}>0$ such that for every pair of sequences of complex numbers $\alpha = \{\alpha_{w}\}$ and $\beta = \{\beta_{z}\}$ satisfying \eqref{seqCond} and every pair of real numbers $M, N>1$, we have 
$$
\left|Q(M, N; \alpha, \beta)\right| \leq C_{\epsilon} \left(M^{\frac{1}{2}}N + M^{2}N^{\frac{3}{4}} + M^{3}N^{\frac{1}{2}}\right)(MN)^{\epsilon}.
$$
\end{lemma}
\begin{proof}
Let $Q(M, N) = Q(M, N; \alpha, \beta)$. Applying the Cauchy-Schwarz inequality to the sum over $z$ and expanding the square, we obtain
$$
|Q(M, N)|^2\leq \sum_{\substack{z\in\BB(N)}}|\beta_z|^2 \suma_{\substack{w_1\in\DD(M)}}\suma_{\substack{w_2\in\DD(M)}}\alpha_{w_1}\overline{\alpha_{w_2}} R(N; w_1, w_2),
$$
where
$$
R(N; w_1, w_2) = \sum_{\substack{z\in\BB(N)}}\gamma(w_1, z)\gamma(w_2, z).
$$
Since $\beta_z$ is bounded in modulus by $N^{\epsilon}$, Lemma~\ref{dav} applied to $L = \mathrm{Id}$ gives
\begin{equation}\label{estBeta}
\sum_{\substack{z\in\BB(N)}}|\beta_z|^2\ll_{\epsilon} N^{\epsilon}\Vol(\BB(N)) + N^{\epsilon}O(\Vol(\partial(\BB(N)))+ 1) \ll_{\epsilon} N^{1+\epsilon}.
\end{equation}
Next, recall that $\gamma(w, z_1) = \gamma(w, z_2)$ whenever $z_1\equiv z_2\bmod \Norm(w)$. Hence we can split the inner sum over $z$ into residue classes modulo $W$. More precisely, if $\zeta = \zeta_1 + \zeta_2\sqrt{2}$, we define $L$ to be the linear transformation $L = W\cdot\mathrm{Id}+(\zeta_1, \zeta_2):\RR^2\rightarrow\RR^2$. Then Lemma~\ref{dav} gives
$$
\begin{array}{ccl}
\displaystyle{R(N; w_1, w_2)} & = & \displaystyle{\sum_{\substack{\zeta\bmod W}}\gamma(w_1, \zeta)\gamma(w_2, \zeta)\sum_{\substack{z\in\BB(N)\\ z\equiv \zeta\bmod W}} 1 } \\
& = & \displaystyle{\sum_{\substack{\zeta\bmod W}}\gamma(w_1, \zeta)\gamma(w_2, \zeta)\left(\frac{2C^2N}{W^2}+O\left(\frac{N^{\frac{1}{2}}}{W} + 1 \right)\right) } \\
& = &\displaystyle{ \frac{2C^2N}{W^2} \sum_{\substack{\zeta\bmod W}}\gamma(w_1, z)\gamma(w_2, z)  + O\left(W^2\left(\frac{N^{\frac{1}{2}}}{W}+1\right)\right)}.
\end{array}
$$
Now set $r = \Norm(\gcd((w_1), (\ow_2)))$. Note that $W\varphi(r)\varphi(W/r)\leq W^2$. Then using Lemma~\ref{keycancellation}, we obtain the estimate
$$
R(N; w_1, w_2)\ll
\begin{cases}
N + WN^{\frac{1}{2}} + W^2 & \text{if }W\text{ and }r\text{ are squares} \\
WN^{\frac{1}{2}} + W^2 & \text{otherwise.}
\end{cases}
$$
By unique factorization in $\ZZT$, the number of primitive elements $w\in\DD$ such that $\Norm(w) = n$ is at most $2^{\omega(n)}\leq\tau(n)\ll_{\epsilon}n^{\epsilon}$. Hence, using the bound $W\ll M^2$ and setting $m_1 = \Norm(w_1)$ and $m_2 = \Norm(w_2)$, we get
\begin{align*}
&|Q(M, N)|^2 \\
& \ll_{\epsilon} N\left(\sumsum_{\substack{m_1, m_2 \leq M \\ m_1m_2=\square}}\left(N + M^2N^{\frac{1}{2}} + M^4\right)+M^2\left(M^2N^{\frac{1}{2}}+M^4\right)\right)(MN)^{\epsilon}.
\end{align*}
We deduce that
$$
Q(M, N)\ll_{\epsilon} \left(M^{\frac{1}{2}}N + M^{\frac{3}{2}}N^{\frac{3}{4}} + M^{\frac{5}{2}}N^{\frac{1}{2}} + M^{2}N^{\frac{3}{4}} + M^{3}N^{\frac{1}{2}}\right)(MN)^{\epsilon},
$$
and the inequality $M\geq 1$ now implies the desired result. 
\end{proof}
The following method, which appears in \cite{FI1}, exploits the multiplicativity of $\gamma(w, z)$ in $z$ to improve the quality of the estimate when $M$ and $N$ are close to each other.
\begin{lemma}\label{generalQsum2}
For every $\epsilon > 0$, there is a constant $C_{\epsilon}>0$ such that for every pair of sequences of complex numbers $\alpha = \{\alpha_{w}\}$ and $\beta = \{\beta_{z}\}$ satisfying \eqref{seqCond} and every pair of real numbers $M, N>1$, we have 
$$
\left|Q(M, N; \alpha, \beta)\right| \leq C_{\epsilon} \left(M^{\frac{11}{12}}N + M^{\frac{7}{6}}N^{\frac{3}{4}} + M^{\frac{4}{3}}N^{\frac{1}{2}}\right)(MN)^{\epsilon}.
$$
\end{lemma}
\begin{proof}
Let $Q(M, N) = Q(M, N; \alpha, \beta)$. We apply H\a"{o}lder's inequality to get
\begin{equation}\label{Q6}
|Q(M, N)|^6 \leq \left(\suma_w|\alpha_w|^{\frac{6}{5}}\right)^5\suma_w\left|\sum_z\beta_z\gamma(w, z)\right|^6.
\end{equation}
By \eqref{gmult}, we can write the second factor above as
\begin{equation}\label{secondfactor}
\suma_w\left|\sum_z\beta_z\gamma(w, z)\right|^6 =: \suma_{\substack{w\in\DD(M)}}\sum_{\substack{z\in\Omega}}\alpha'_w\beta'_z\gamma(w, z),
\end{equation}
where $\alpha'_w = \mul(w)^5$ and
$$
\beta'_z = \sum_{\substack{z_1\cdots z_6 = z \\ z_1, \ldots, z_6 \in\BB(N)}}\beta_{z_1}\overline{\beta_{z_2}}\cdots\beta_{z_{5}}\overline{\beta_{z_6}}.
$$
Note that $\beta_z$ is supported on $z\in\BB(27C^6N^3)$. Now using Lemma~\ref{generalQsum1} to estimate the sum \eqref{secondfactor}, and substituting back into \eqref{Q6}, we obtain the desired result.
\end{proof}
The final step is to exploit the symmetry of the symbol $\gamma(w, z)$ coming from its reciprocity law. Suppose that $w = a + b\sqrt{2}\in\ZZT$ is primitive and totally positive. By \eqref{gammaJacobi} and the law of quadratic reciprocity, we have
$$
\mul(w) = \gamma(w, 1) = \left(\frac{a}{a^2-2b^2}\right) = (-1)^{\frac{a-1}{2}\cdot\frac{a^2-2b^2-1}{2}}\left(\frac{-2}{a}\right),
$$
and so $\mul(w)\in\{\pm 1\}$ depends only on the residue class of $w$ modulo $8\ZZT$. Lemma~\ref{reciprocitylem} then implies that for every pair of odd and totally positive $w, z\in \ZZT$, we have
$$
\gamma(w, z)=\delta\cdot \gamma(z, w),
$$
where $\delta = \delta(w\bmod 8, z\bmod 8) := \mul(wz) \in\{\pm 1\}$ depends only on the congruence classes of $w$ and $z$ modulo $8\ZZT$. We are thus led to decompose the sum $Q(M, N; \alpha, \beta)$ as
$$
Q(M, N; \alpha, \beta) = \sumsum_{\substack{w_0\bmod 8 \\ z_0\bmod 8}}Q(M, N; \alpha(w_0), \beta(z_0)),
$$
where $\alpha(w_0)$ and $\beta(z_0)$ are sequences indexed by non-zero elements of $\ZZT$ defined by
$$
\alpha(w_0)_w := \alpha_w\cdot\ONE(w\equiv w_0\bmod 8)
$$
and
$$
\beta(z_0)_z := \beta_z\cdot\ONE(z\equiv z_0\bmod 8).
$$
Here $\ONE(P)$ is the indicator function of a property $P$. We will now prove
\begin{lemma}\label{generalQsum3}
For every $\epsilon > 0$, there is a constant $C_{\epsilon}>0$ such that for every pair of sequences of complex numbers $\alpha = \{\alpha_{w}\}$ and $\beta = \{\beta_{z}\}$ such that \eqref{seqCond} holds and such that $\beta$ is supported on primitive $z\in\DD(N)$, and for every pair of real numbers $M, N>1$, we have 
$$
\left|Q(M, N; \alpha, \beta)\right| \leq C_{\epsilon} \left(M+N\right)^{\frac{1}{12}}(MN)^{\frac{11}{12}+\epsilon}.
$$
\end{lemma}
\begin{proof}
It suffices to establish the desired estimate for the sequences $\alpha(w_0)$ and $\beta(z_0)$ for each pair of congruence classes $w_0$ and $z_0$ modulo $8\ZZT$. So fix congruence classes $w_0$ and $z_0$ modulo $8\ZZT$. Note that the sum $Q(N, M; \beta(z_0), \alpha(w_0))$ satisfies the assumptions of Lemma~\ref{generalQsum2}. Thus, applying Lemma~\ref{generalQsum2} gives the estimate
\begin{equation}\label{estQsum31}
Q(M, N; \alpha(w_0), \beta(z_0)) \ll_{\epsilon} \left(M^{\frac{11}{12}}N + M^{\frac{7}{6}}N^{\frac{3}{4}} + M^{\frac{4}{3}}N^{\frac{1}{2}}\right)(MN)^{\epsilon}.
\end{equation}
As discussed above, by Lemma~\ref{reciprocitylem}, we have
$$
Q(M, N; \alpha(w_0), \beta(z_0)) = \delta(w_0, z_0)\cdot Q(N, M; \beta(z_0), \alpha(w_0)).
$$
Applying Lemma~\ref{generalQsum2} to the right-hand side above, we also get
\begin{equation}\label{estQsum32}
Q(M, N; \alpha(w_0), \beta(z_0)) \ll_{\epsilon} \left(N^{\frac{11}{12}}M + N^{\frac{7}{6}}M^{\frac{3}{4}} + N^{\frac{4}{3}}M^{\frac{1}{2}}\right)(MN)^{\epsilon}.
\end{equation}
Finally, taking the minimum of the terms in \eqref{estQsum31} and \eqref{estQsum32} in the appropriate ranges, we obtain
$$
Q(M, N; \alpha(w_0), \beta(z_0)) \ll_{\epsilon} \left(M^{\frac{11}{12}}N+N^{\frac{11}{12}}\right)(MN)^{\epsilon}, 
$$
and then the inequality
$$
M^{\frac{1}{12}}+N^{\frac{1}{12}}\leq 2\max\{M, N\}^{\frac{1}{12}}\leq 2(M+N)^{\frac{1}{12}}
$$
gives the desired result.
\end{proof}
\subsection{Twisted multiplicativity of governing symbols}
Recall that if $u+v\sqrt{2}$ is a totally positive odd element of $\ZZT$, we defined the governing symbol $[u+v\sqrt{2}]$ to be
$$[u+v\sqrt{2}] = \left(\frac{v}{u}\right).$$
Thus $[u+v\sqrt{2}] = 0$ whenever $u+v\sqrt{2}$ is not primitive.
\\\\
A key feature of the governing symbol $[\cdot]$ which leads to significant cancellation in \eqref{sumB} is that $[\cdot]$ is \textit{not} multiplicative, i.e., the relation $[wz] = [w][z]$ does \textit{not} hold for all totally positive $w$ and $z$. Instead, the equation above becomes essentially valid when twisted by $\gamma(w, z)$. We now state our result more precisely. 
\\\\
We now introduce notation that will simplify the subsequent arguments. Suppose that $f_1$ and $f_2$ are functions $\ZZ^r\rightarrow\CC$. For $x\in \ZZ^r$, we write $f_1\sim f_2$ (or more conveniently $f_1(x)\sim f_2(x)$) if there exists a function $\delta: \ZZ^r\rightarrow\{\pm 1\}$ such that $\delta$ factors though $(\ZZ/16\ZZ)^r$, i.e., the value of $\delta(x)$ depends only on the congruence classes of the coordinates of $x$ modulo $16$, and such that
$$
f_1(x) = \delta(x)f_2(x)
$$
for all $x\in\ZZ^r$. For instance, $[u+v\sqrt{2}]_{\phi, \psi}\sim[u+v\sqrt{2}]_{\phi', \psi'}$ for any four Dirichlet characters $\phi$, $\psi$, $\phi'$, $\psi'$ modulo $16$.
\\\\
The following proposition is analogous to \cite[Lemma 20.1, p.\ 1021]{FI1}. It is perhaps the most surprising part of the proof of Proposition~\ref{propB}.    
\begin{prop}\label{almostMult}
Let $w = a+b\sqrt{2}$ and $z = c+d\sqrt{2}$ be two primitive, totally positive, odd elements of $\ZZT$. Then
$$
[wz] \sim [w][z]\gamma(w, z).
$$
\end{prop}
\begin{proof}
When $wz$ is not primitive, then $[wz] = 0$ and $\gamma(w, z) = 0$, and so the result follows. Hence we may assume that $wz$ is primitive.
\\\\
First note that
$$
wz = (ac+2bd)+(ad+bc)\sqrt{2}.
$$
We set $\rho = (a, d)$ and define $a_1$ and $d_1$ by the equalities $a = \rho a_1$ and $d = \rho d_1$, respectively. Then
$$
[wz] = \left(\frac{ad+bc}{ac+2bd}\right) = \left(\frac{ad+bc}{\rho}\right)\left(\frac{ad+bc}{a_1c+2bd_1}\right),
$$
and since $\rho$ divides $ad$, the above simplifies to
$$
[wz] = \left(\frac{bc}{\rho}\right)\left(\frac{ad+bc}{a_1c+2bd_1}\right).
$$
Now, since $w$ is primitive, $a_1$ is relatively prime to $b$ and hence also to $a_1c+2bd_1$. Hence we may write
$$
c\equiv -2bd_1/a_1 \pmod{a_1c+2bd_1},
$$
so that the second factor in the expression above becomes
$$
\left(\frac{ad+bc}{a_1c+2bd_1}\right) = \left(\frac{ad-2b^2d_1/a_1}{a_1c+2bd_1}\right) = \left(\frac{a_1d_1}{a_1c+2bd_1}\right)\left(\frac{\rho^2 - 2b^2/a_1^2}{a_1c+2bd_1}\right).
$$
As $a^2-2b^2 = a_1^2(\rho^2-2b^2/a_1^2)$, we deduce that
$$
[wz]\sim \left(\frac{bc}{\rho}\right)\left(\frac{a_1d_1}{a_1c+2bd_1}\right)\left(\frac{a^2-2b^2}{a_1c+2bd_1}\right).
$$
We write the last factor in the expression above as
$$
\left(\frac{a^2-2b^2}{a_1c+2bd_1}\right) = \left(\frac{a^2-2b^2}{\rho}\right)\left(\frac{a^2-2b^2}{ac+2bd}\right),
$$
and use the fact that
$$
\left(\frac{a^2-2b^2}{\rho}\right) =\left(\frac{-2b^2}{\rho}\right)= \left(\frac{-2}{\rho}\right)
$$
to conclude that
$$
[wz]\sim \left(\frac{-2bc}{\rho}\right)\left(\frac{a_1d_1}{a_1c+2bd_1}\right)\left(\frac{a^2-2b^2}{ac+2bd}\right).
$$
The law of quadratic reciprocity implies that
$$
\left(\frac{a^2-2b^2}{ac+2bd}\right)\sim \left(\frac{ac+2bd}{a^2-2b^2}\right),
$$
so that, by \eqref{gammaJacobi},
$$
[wz]\sim \left(\frac{-2bc}{\rho}\right)\left(\frac{a_1d_1}{a_1c+2bd_1}\right)\gamma(w, z)
$$
We again use the law of quadratic reciprocity to treat the middle term above. We get
$$
\left(\frac{a_1}{a_1c+2bd_1}\right) = (-1)^{\nu_1(a, b, c, d, \rho)}\left(\frac{2}{a_1}\right) \left(\frac{bd_1}{a_1}\right),
$$
where 
$$
\nu_1(a, b, c, d, \rho)\equiv \frac{a_1-1}{2}\cdot\frac{r_1 - 1}{2}\bmod 2
$$
and
$$
r_1 = a_1c+2bd_1.
$$
Similarly, we write $d_1$ as 
$$
d_1 = 2^ed_2,
$$
where $d_2$ is odd, and compute that
$$
\left(\frac{d_1}{a_1c+2bd_1}\right) = (-1)^{\nu_2(a, b, c, d, \rho)} \left(\frac{d_1}{a_1c}\right),
$$
where now
$$
\nu_2(a, b, c, d, \rho)\equiv e\frac{r_1^2-1}{8}+\frac{d_2-1}{2}\cdot\frac{r_1 - 1}{2}+\frac{d_2-1}{2}\cdot\frac{a_1c - 1}{2}+e\frac{a_1^2c^2-1}{8}\bmod 2.
$$
We thus have
$$
[wz] \sim (-1)^{\nu_1+\nu_2}\left(\frac{2}{a_1}\right) \left(\frac{-2bc}{\rho}\right)\left(\frac{b}{a_1}\right)\left(\frac{d_1}{c}\right)\gamma(w, z),
$$
which simplifies to
$$
[wz]\sim (-1)^{\nu_1+\nu_2+\nu_3}\left(\frac{-1}{\rho}\right)\left(\frac{b}{a}\right)\left(\frac{d}{c}\right)\gamma(w, z),
$$
where 
$$
\nu_3 = \nu_3(c, \rho)\equiv \frac{\rho-1}{2}\cdot\frac{c-1}{2}\bmod 2.
$$
It remains to show that
$$
(-1)^{\nu_1+\nu_2+\nu_3}\left(\frac{-1}{\rho}\right)
$$
depends only on the residue classes of $a, b, c, d$ modulo $16$. First note that whether $e = 0$, $e = 1$, or $e\geq 2$ depends only on the residue class of $d$ modulo $4$ (and hence also modulo $16$). Hence we can split into cases $e = 0$, $e = 1$, and $e\geq 2$.
\\\\ 
Note that if $e\geq 2$ or $e = 1$ and $b\equiv 0\bmod 2$, then $r_1 \equiv a_1c\bmod 8$. Using this observation and the definitions of $\nu_1$, $\nu_2$, and $\nu_3$, we find that
$$
\nu_2\equiv
\begin{cases}
\frac{d_1-1}{2} \bmod 2& \text{if }e = 0\text{ and }b\equiv 1\bmod {2}\\
1 \bmod 2 & \text{if }e = 1\text{ and }b\equiv 1\bmod {2}\\
0 \bmod 2& \text{otherwise.}
\end{cases}
$$ 
First suppose $e\geq 2$. Then $r_1 \equiv a_1c\bmod 8$ and $\nu_2 \equiv 0\bmod 2$. Suppose first that $c\equiv 1\bmod 4$. Then $\nu_3 \equiv 0\bmod 2$ as well. Moreover, $a_1\equiv r_1\bmod 4$, so that 
$$
\nu_1\equiv \frac{a_1-1}{2}\cdot\frac{a_1-1}{2}\equiv \frac{a_1-1}{2} \bmod 2.
$$
Finally, as $a = a_1\rho$,
$$
\left(\frac{-1}{a}\right) = \left(\frac{-1}{a_1}\right) \left(\frac{-1}{\rho}\right)  
$$
and so $\nu_1+(\rho-1)/2\equiv (a-1)/2\bmod 2$. Now suppose $c\equiv 3\bmod 4$. Then $\rho$ and $c\rho$ are odd and different modulo $2$, and so $\nu_3+(\rho-1)/2\equiv 1 \bmod 2$. Moreover, $r_1\equiv 3a_1\bmod 4$, so that $r_1$ and $a_1$ are odd and different modulo $4$. Hence at least one of $(r_1-1)/2$ and $(a_1-1)/2$ is $0 \bmod 2$ and so $\nu_1 = 0$. Collecting these results, we get  
$$
\nu_1+\nu_2+\nu_3+\frac{\rho-1}{2}\equiv
\begin{cases}
\frac{a-1}{2} \bmod 2& \text{if }c\equiv 1\bmod {4}\\
1 \bmod 2& \text{if }c\equiv 3\bmod {4}.
\end{cases}
$$
Now suppose $e = 1$. Then splitting into cases similarly as above, we get 
$$
\nu_1+\nu_2+\nu_3+\frac{\rho-1}{2}\equiv
\begin{cases}
\frac{a-1}{2} \bmod 2& \text{if }b\equiv 0\bmod 2\text{ and }c\equiv 1\bmod {4}\\
0 \bmod 2& \text{if }b\equiv 0\bmod 2\text{ and }c\equiv 3\bmod {4}\\
\frac{a-1}{2}+1 \bmod 2& \text{if }b\equiv 1\bmod 2\text{ and }c\equiv 1\bmod {4}\\
1 \bmod 2& \text{if }b\equiv 1\bmod 2\text{ and }c\equiv 3\bmod {4}.
\end{cases}
$$
Finally, suppose $e = 0$. Then
$$
\nu_1+\nu_2+\nu_3+\frac{\rho-1}{2}\equiv
\begin{cases}
\frac{a-1}{2} \bmod 2& \text{if }b\equiv 0\bmod 2\text{ and }c\equiv 1\bmod {4}\\
0 \bmod 2& \text{if }b\equiv 0\bmod 2\text{ and }c\equiv 3\bmod {4}\\
\frac{d-1}{2} \bmod 2& \text{if }b\equiv 1\bmod 2\text{ and }c\equiv 1\bmod {4}\\
\frac{a-1}{2}+\frac{d-1}{2} \bmod 2& \text{if }b\equiv 1\bmod 2\text{ and }c\equiv 3\bmod {4}.
\end{cases}
$$
This proves the lemma.
\end{proof}

\subsection{Proof of Proposition~\ref{propB}}
We are now ready to conclude the proof of Proposition~\ref{propB}. The bilinear sum \eqref{sumB} can be written as
$$B(M, N) = \sum_{k = 0}^3B_k(M, N),$$
where
\begin{equation}\label{sumBk}
B_k(M, N) = \sum_{w\in\DD(M)}\sum_{z\in\DD(N)}\alpha_w\beta_z[\varepsilon^{2k}wz]_{\phi, \psi}.
\end{equation}
Here $\alpha_w = \alpha_{(w)}$ and $\beta_{z} = \beta_{(z)}$, i.e., $\alpha_w$ (resp. $\beta_z$) depends only on the ideal generated by $w$ (resp. $z$).
\\\\
It is enough to estimate \eqref{sumBk} for each $0\leq k\leq 3$. First, suppose $u+v\sqrt{2}\succ 0$ is primitive and odd. Then by Proposition~\ref{almostMult}, we have 
$$
[\varepsilon^{2k}(u+v\sqrt{2})]\sim [u+v\sqrt{2}][\varepsilon^{2k}]\gamma(\varepsilon^{2k}, u+v\sqrt{2})\sim [u+v\sqrt{2}].
$$
We write $w = a+b\sqrt{2}$ and $z=c+d\sqrt{2}$ and split \eqref{sumBk} into $8^2\cdot 16^2$ sums by fixing congruence classes of $a$, $b$, $c$, and $d$ modulo $16$ (where the congruence classes of $a$ and $c$ are invertible). Then it suffices to estimate each sum
$$\sum_{\substack{w\in\DD(M)\\ w\equiv w_0\bmod 16}}\sum_{\substack{z\in\DD(N)\\ z\equiv z_0\bmod 16}}\alpha_w\beta_z[wz].$$
Unless both $w$ and $z$ are primitive, $wz$ is not primitive, and hence $[wz] = 0$ . Using Proposition~\ref{almostMult} again and replacing $\alpha_w$ by $\alpha_w[w]\ONE(w\equiv w_0\bmod 16)$ and $\beta_z$ by $\beta_z[z]\ONE(w\equiv w_0\bmod 16)$, it now suffices to estimate sums of the type
$$
\suma_{\substack{w\in\DD(M)}}\suma_{\substack{z\in\DD(N)}}\alpha_w\beta_z\gamma(w, z).
$$
This is exactly a sum of the type $Q(M, N; \alpha, \beta)$ as in Lemma~\ref{generalQsum3}, and so Proposition~\ref{propB} follows. This completes the proof of Theorem~\ref{mainThm2} and hence also Theorem~\ref{mainThm}.

\section{Counting primes}\label{Countingprimes}
In this section we give evidence that a governing field for the $16$-rank of the family $\{\QQP\}_{p\equiv 3(4)}$ does \textit{not} exist. To explain why, we first define a prime counting function. Suppose $M/\QQ$ is a normal extension. Let $S$ be a subset of $\Gal(M/\QQ)$ which is a union of conjugacy classes. We define
$$
\pi(M, S, X): = \#\{p\leq X:\text{the Artin class of }p\text{ in }\Gal(M/\QQ) \text{ is a subset of } S\}
$$
Given any normal extension $M/\QQ$ of degree $d$ and a subset $S$ of $\Gal(M/\QQ)$ stable under conjugation, the \v{C}ebotarev Density Theorem using the best known zero-free regions of $L$-functions gives \cite[Th\a'{e}or\a`{e}me 2, p.\ 132]{JPS}, for some constant $c>0$,
$$
\pi(M, S, X) = \frac{\# S}{\# \Gal(M/\QQ)}\mathrm{Li}(X) + O(\#S X\exp(-cd^{-1/2}\log^{1/2}X)).
$$
Hence given any two subsets $S_1$ and $S_2$ of $\Gal(M/\QQ)$ which are stable under conjugation and of the same cardinality,
$$
\pi(M, S_1, X) - \pi(M, S_2, X) \ll \#S_1 X\exp(-cd^{-1/2}\log^{1/2}X)
$$
is the best known bound. Note that this bound is weaker than $X^{1-\delta}$ for any $\delta>0$. For instance, it is \textit{not} known if
$$
\#\left\{p\leq X \text{ prime}:\ p\equiv 1\bmod 4\right\} - \#\left\{p\leq X \text{ prime}:\ p\equiv -1\bmod 4\right\} \ll X^{0.9999}.
$$
However, we have the following result.
\begin{theorem}\label{countPrimes}
Suppose that there exists a governing field $M$ for the $16$-rank of the family $\{\QQP\}_{p\equiv 3(4)}$. Then there exist disjoint subsets $S_1$ and $S_2$ of $\Gal(M/\QQ)$ which are stable under conjugation and of equal size such that 
$$\pi(M, S_1, X) - \pi(M, S_2, X) \ll X^{\frac{199}{200}}$$
\end{theorem}
\begin{proof}
We simply let $S_1$ be the union of Artin classes $c_p$ for primes $p$ satisfying $\rk_{16}\CL(-8p) = 1$ and $S_2$ be the union of Artin classes $c_p$ for primes $p$ satisfying $\rk_8\CL(-8p) = 1$ but $\rk_{16}\CL(-8p) = 0$. The result now immediately follows from Theorem~\ref{mainCor}.   
\end{proof}
However, with our current methods of complex analysis applied to $L$-functions, we are not able to produce an error term of the form $O(x^{1-\delta_M})$ for any $\delta_M>0$. This leads us to believe that a governing field $M$ for the $16$-rank of the family $\{\QQP\}_{p\equiv 3(4)}$ is unlikely to exist.

\bibliographystyle{plain}
\bibliography{sixteenrankreferencesQ}
\end{document}